\documentclass{article}%
\usepackage{authblk}
\usepackage{mleftright}
\usepackage{color}
\usepackage{amsfonts}
\usepackage{amsmath}
\usepackage{amssymb}
\usepackage{amsthm}
\usepackage{fullpage}
\usepackage{graphicx}
\usepackage{tikz}
\usepackage{scalefnt}
\usepackage{xcolor}
\allowdisplaybreaks
\providecommand{\U}[1]{\protect\rule{.1in}{.1in}}
\newtheorem{theorem}{Theorem}
\newtheorem*{theorem*}{Theorem}
\newtheorem{lemma}[theorem]{Lemma}
\newtheorem{proposition}[theorem]{Proposition}

\newtheorem{definition}[theorem]{Definition}
\newtheorem{notation}[theorem]{Notation}
\newtheorem{algorithm}[theorem]{Algorithm}
\newtheorem{example}[theorem]{Example}
\newtheorem{remark}[theorem]{Remark}
\newtheorem{problem}[theorem]{Problem}

\DeclareMathOperator{\sign}{sign}

\DeclareMathOperator{\tdeg}{taildeg}

\DeclareMathOperator{\diag}{diag}
\newcommand{\Ay}{A_{\rm{sym}}}
\newcommand{\Cy}{C_{\rm{sym}}}
\newcommand{\EC}{{\rm{EC}}}
\newcommand{\GEC}{{\rm{GEC}}}

\newcommand\xqed[1]{%
  \leavevmode\unskip\penalty9999 \hbox{}\nobreak\hfill
  \quad\hbox{#1}}
\newcommand\exend{\xqed{$\triangle$}}
\author[1]{Hoon Hong\thanks{hong@ncsu.edu}}
\author[1]{Daniel Profili\thanks{daprofil@ncsu.edu, corresponding author}}
\author[2]{J. Rafael Sendra\thanks{jrafael.sendra@cunef.edu}}

\affil[1]{Department of Mathematics, North Carolina State University, USA}

\affil[2]{Department of Mathematics, CUNEF Universidad, Spain}

\begin{document}

\title{Conditions for eigenvalue configuration \\ of two real symmetric matrices \\ (Symmetric polynomial approach)\thanks{A special case of the result here, without proofs, was presented as a poster at ISSAC 2024 with an abstract in \cite{issac} }}
\maketitle

\begin{abstract}
Given two real symmetric matrices, their eigenvalue configuration is the
relative arrangement of their eigenvalues on the real line.
In this paper, we consider the following problem: given two parametric real symmetric matrices and an eigenvalue configuration, find a
simple condition on the parameters such that their eigenvalues have
the given configuration.
In this paper, we give an algorithm which expresses the eigenvalue configuration problem as a real root counting problem of certain symmetric polynomials, whose roots can be counted using the Fundamental Theorem of Symmetric Polynomials and Descartes' rule of signs.
\end{abstract}






\section{Introduction}
For two real symmetric matrices, their eigenvalue configuration is the relative arrangement of their eigenvalues on the real line.
In this paper, we consider the \textit{eigenvalue configuration problem}: given two parametric real symmetric matrices $F$ and $G$ and an eigenvalue
configuration, produce quantifier-free conditions on the entries of
$F$ and $G$ so that their eigenvalues are arranged in the given way. For an alternative solution to the same problem, see our related work \cite{signature}.

A fundamental problem in computational algebra and geometry, called the real root counting problem, is to 
determine
a quantifier-free condition on the coefficients of a
polynomial such that its roots lie in a given subset of the plane. This is a very general problem which appears in many different areas, including algebraic geometry \cite{Furchi2025}, complex analysis \cite{gati2023degree6hyperbolicpolynomials}, and graph theory \cite{bartzos2021newupperboundsnumber}, among others.
The eigenvalue configuration problem is a special case of this problem where we desire to count the roots of one polynomial that lie within intervals determined by the roots of another polynomial.
In particular, the eigenvalue configuration problem generalizes Descartes' rule of signs, which is a fundamental tool in algebraic geometry which is still used today in many fields (see e.g. \cite{pastuszak2024oneparameterfamilieshermiticitypreservingsuperoperators}, \cite{deshpande2024existenceuniquesolutionparametrized}).
Recall that Descartes' rule of signs states that, for a univariate polynomial
$g$ with real coefficients, the number of
positive real roots of $g$, counted with multiplicity, is bounded above by the sign variation count of the coefficients of
$g$; i.e., the number of times consecutive coefficients change sign, ignoring
zeros. In the case where $g$ has only real roots, then the number of positive roots is counted exactly by the sign variation count of the coefficients. Descartes' rule of signs could also be seen as determining the arrangement of the eigenvalues of the matrix $F = [0]$ and a real matrix $G$ whose characteristic polynomial is $g$.
The eigenvalue configuration problem therefore extends Descartes' rule of signs by
allowing two polynomials of arbitrary degrees; in addition, we reframe the problem slightly by
considering characteristic polynomials of real symmetric matrices, as these occur
naturally in many areas.

Since Descartes' rule of signs is widely used, it is natural to expect that a generalization will have many applications. For
example, one could use this generalization in investigating the impact on the eigenvalues under low rank updates~\cite{benaych2011eigenvalues,hill1992refined}. In our related work \cite{signature}, we discuss a few more potential applications. 

The main difficulty of the eigenvalue configuration problem comes from the fact that it is not practically solvable using existing methods.
While the eigenvalue configuration problem can be solved using general quantifier elimination algorithms 
(see e.g. \cite{McCallum:84,
Ben-Or_Kozen_Reif:86, Grigorev:88,Weispfenning:95,
Gonzalez_Lombardi_Recio_Roy:89, Hong:90a, Hong:90b, Collins_Hong:91, Hong:92a,
Renegar:92a,Canny:93a,Loos_Weispfenning:93, Weispfenning:94b,
Gonzalez-Vega:96, DBLP:journals/jsc/McCallum99, DBLP:conf/issac/McCallumB09,
DBLP:journals/jsc/Brown01a, DBLP:journals/jsc/Brown01, Strzebonski06,
DBLP:conf/issac/ChenMXY09, DBLP:journals/jsc/Brown12, Hong_Safey:2012}), it is very inefficient.
Furthermore, the outputs of these algorithms grow very quickly toward being incomprehensible
for even moderately sized inputs. As a consequence of these limitations, we must exploit particular properties of the eigenvalue configuration problem to develop a practical solution.


The main contribution of this paper is to provide an efficient and structured solution to the eigenvalue configuration problem. We accomplish this by defining combinatorial objects related to the eigenvalue configuration of the matrices which can be counted as the roots of certain symmetric polynomials. In our related work \cite{signature}, we approach the same problem via a method based on the theory of the signature of matrices.

The contribution of this paper serves as one possible way of generalizing Descartes' rule of signs. There has been recent work on generalizing Descartes' rule of signs to consider single multivariate polynomials \cite{telek2024geometry}, but we are not aware of any previous attempts to generalize to more than one univariate polynomial.

The non-triviality of the eigenvalue configuration problem comes from the fact that there exists no closed form  expression for the eigenvalues of a general matrix in terms of the entries. Additionally, while the eigenvalue configuration of an arbitrary given pair of numeric matrices can readily be computed via numeric methods, it is not possible to use numeric methods to solve the problem parametrically, as we will do here.



The paper is structured as follows. In~Section~\ref{sec:EC}, we discuss and define the eigenvalue configuration of two real symmetric matrices. 
In~Section~\ref{sec:problem}, we define and state the problem precisely. 
In~Section~\ref{sec:main}, we state our main theorem (Theorem \ref{thm:main}).
In~Section~\ref{sec:proof}, we prove the main theorem. 
The proof is divided into two parts. 
First, in Section \ref{sec:generic}, we prove it  for special inputs  (``generic'' pairs of real symmetric matrices).
and then, in Section \ref{sec:arbitrary}, we prove it  for all  inputs (arbitrary pair of real symmetric matrices). 
In~Section~\ref{sec:algorithms}, we render our main result into an algorithm for those who are  interested in
implementation. 
In~Section~\ref{sec:conclusion}, we summarize the contributions 
of this paper and discuss potential future directions.
\section{Problem}

\subsection{Defining eigenvalue configuration}
\label{sec:EC}

In this section, we give a precise definition of ``eigenvalue configuration.'' For a more thorough explanation, see our related work \cite{signature}, which deals with the same problem.
First, we introduce some notation.
\begin{notation}
\label{basic_notation}\ \ 
\end{notation}

\begin{enumerate}
\item Let $F = [a_{ij}] \in\mathbb{R}^{m \times m}$ and $G = [b_{ij}] \in\mathbb{R}^{n \times n}$ be
real symmetric matrices.

\item Let $\alpha=\left(  \alpha_{1},\ldots,\alpha_{m}\right)  $ be the
eigenvalues of$\ F$.

Let $\beta=\left(  \beta_{1},\ldots,\beta_{n}\right)  $ \ be the eigenvalues
of $G$.

Since $F$~and~$G$ are real symmetric, all their eigenvalues are real. Thus without losing
generality, let us index the eigenvalues so that
$\alpha_{1}\le\alpha_{2}\le\cdots\le\alpha_{m}$ and $\beta_{1}  \le\beta_{2}\le\cdots\le\beta_{n}$.

\item Let $A_{t}$ denote the set $\{x \in \mathbb{R}: \alpha_{t} < x < \alpha_{t+1}\}$ for $t=1,\ldots,m$, where
$\alpha_{m+1}=\infty$. Note that this set could be empty if $\alpha_t = \alpha_{t+1}$.
\end{enumerate}
We first consider the eigenvalue configuration when the pair $F$ and $G$ is ``generic'' in the following sense.

\begin{definition}
[Generic]\label{def:generic} We say that the pair of matrices $F$ and $G$ is \textbf{generic} if
$F$ and $G$ do not share any eigenvalues.
\end{definition}


\begin{definition}
[Eigenvalue configuration of generic pairs of matrices]\label{def:GEC}The eigenvalue configuration of the generic pair of matrices~$F$ and~$G$, written as $\GEC\left(  F,G\right)  $, is the
tuple $$c=(c_{1},\dots,c_{m})$$ where $$c_{t}=\#\{i:\beta_{i}\in A_{t}\}.$$ 
\end{definition}

\begin{example}
Let $F \in\mathbb{R}^{6 \times6}$ and $G \in\mathbb{R}^{4 \times4} $ be symmetric
matrices such that their corresponding eigenvalues are
\[
\alpha=(0,0,1,3,5,8),\qquad\beta=(-1,2,2,6).
\]
The eigenvalues are arranged on the real line as follows.
\[
\begin{tikzpicture}
\draw (-2, 0) -- (9, 0);
\fill [red]   (0,0) circle(3pt) node[label=$\alpha_1$]{};
\fill [red]   (0,0.8) circle(3pt) node[label=$\alpha_2$]{};
\fill [red]  (1,0) circle(3pt) node[label=$\alpha_3$]{};
\fill [red]   (3,0) circle(3pt) node[label=$\alpha_4$]{};
\fill [red]  (5,0) circle(3pt) node[label=$\alpha_5$]{};
\fill [red]  (8,0) circle(3pt) node[label=$\alpha_6$]{};
\fill [blue] (-1,0) circle(3pt) node[label=$\beta_1$ ]{};
\fill [blue]  (2,0) circle(3pt) node[label=$\beta_2$ ]{};
\fill [blue]  (2,0.8) circle(3pt) node[label=$\beta_3$ ]{};
\fill [blue]  (6,0) circle(3pt) node[label=$\beta_4$]{};
\end{tikzpicture}
\]

\medskip

\noindent Then we have
\[%
\begin{array}
[c]{lcl}%
A_{1}=(\alpha_{1},\alpha_{2}) &= \emptyset& \\
A_{2}=(\alpha_{2},\alpha_{3}) && \\
A_{3}=(\alpha_{3},\alpha_{4}) && \ni\beta_2, \beta_3\\
A_{4}=(\alpha_{4},\alpha_5) && \\
A_{5}=(\alpha_{5},\alpha_6) && \ni \beta_4\\
A_{6}=(\alpha_{6},\infty) && 
\end{array}
\]

\noindent Therefore
\[
\GEC(F,G)=(0,0,2,0,1,0).
\]
Note that the sum of the entries in this vector is the total number of eigenvalues of $G$ minus the number of eigenvalues of $G$ which lie to the left of $\alpha_1$.
\exend
\end{example}

\noindent Now, we extend the above definition to pairs of matrices which are not generic.

\begin{definition}
[Eigenvalue configuration of arbitrary pairs of matrices]\label{def:EC}Let $F$~and~$G$ be arbitrary real symmetric matrices. The eigenvalue configuration of~$F$ and~$G$, written as $\EC\left(  F,G\right)  $, is defined as

\begin{align*}
  \EC(F,G)  
          &= \frac{1}{2^m} \sum_{d \in \{-1, 1\}^m} \GEC \left( F_d ,G
            \right)
\end{align*}
where 
\begin{align*}
  \varepsilon &= \text{ ``small enough'' positive number (see Remark below)} \\
  F_d &= \text{matrix with eigenvalues $\alpha_i + \varepsilon d_i$ for $i = 1, \dots, m$.}
\end{align*}
\end{definition}

\begin{remark}
  \label{remark:generalEC}
  Some notes on Definition \ref{def:EC}:
  \begin{enumerate}
  \item The number $\varepsilon$ is ``small enough'' if no eigenvalue of $F$ is perturbed far enough to cross over the next or previous eigenvalue. In practice, one could take $\varepsilon$ to be smaller than the minimum distance between distinct elements of $(\alpha_1, \dots, \alpha_m, \beta_1, \dots, \beta_n)$.
  \item Note that the pairs $F_d$ and $G$ are generic for all $d \in \{-1, 1\}^m$.
  \item If the pair $F$~and~$G$ is generic, then $\GEC(F_d, G) = \GEC(F,G)$.
  \end{enumerate}
\end{remark}

  \begin{example}[Eigenvalue configuration of arbitrary matrices] \label{ex:g1}
    Let $F \in \mathbb{R}^{3 \times 3}$ and $G \in \mathbb{R}^{2 \times 2}$ be real symmetric matrices such that their corresponding eigenvalues are
    \[
      \alpha = (0,1,1), \qquad \beta = (1,2).
    \]
    Pictorially, the eigenvalue arrangement is as shown in Figure \ref{fig:exec}.
    \begin{figure}[h!]
      \centering
      \begin{tikzpicture}
        \draw (-1, 0) -- (3, 0);
        \fill [red] (0,0) circle(3pt) node[label=$\alpha_1$]{}; \fill [red]
        (1,0) circle(3pt) node[label=$\alpha_2$]{}; \fill [red] (1,0.75)
        circle(3pt) node[label=$\alpha_3$]{}; \fill [blue] (1,1.5)
        circle(3pt) node[label=$\beta_1$]{}; \fill [blue] (2,0)
        circle(3pt) node[label=$\beta_2$]{};
      \end{tikzpicture}
      
      \caption{Eigenvalue configuration with shared eigenvalues}
      \label{fig:exec}
    \end{figure}
    Let $\varepsilon = 0.5$ (which can be verified to satisfy Definition \ref{def:EC}).
  Now, for each $d \in \{-1, 1\}^3$, we compute the eigenvalue configuration of $F_d$ and $G$ in the tables below.
  Note that each pair $F_d$ and $G$ is generic.
  \[
    \begin{array}{|c|c|c|}
      \hline
      d & \text{Picture} & \GEC(F_d, G) \\
      \hline
      - - - 
                & 
                  \begin{tikzpicture}[scale=0.5]
                    \draw (-1, 0) -- (3, 0);
                    \fill [red] (-0.5,0) circle(4pt) node[]{}; \fill
                    [red] (0.5,0) circle(4pt) node[]{}; \fill [red]
                    (0.5,0.5) circle(4pt) node[]{}; \fill [blue] (1,0)
                    circle(4pt) node[]{}; \fill [blue] (2,0)
                    circle(4pt) node[]{};
                  \end{tikzpicture}
                                 &
                                   \begin{bmatrix}
                                     0 \\ 0 \\ 2
                                   \end{bmatrix} \\
      - - + 
                & 
                  \begin{tikzpicture}[scale=0.5]
                    \draw (-1, 0) -- (3, 0);
                    \fill [red] (-0.5,0) circle(4pt) node[]{}; \fill
                    [red] (0.5,0) circle(4pt) node[]{}; \fill [red]
                    (1.5,0) circle(4pt) node[]{}; \fill [blue] (1,0)
                    circle(4pt) node[]{}; \fill [blue] (2,0)
                    circle(4pt) node[]{};
                  \end{tikzpicture}
                                 &
                                   \begin{bmatrix}
                                     0 \\ 1 \\ 1
                                   \end{bmatrix} \\
      - + - 
                & 
                  \begin{tikzpicture}[scale=0.5]
                    \draw (-1, 0) -- (3, 0);
                    \fill [red] (-0.5,0) circle(4pt) node[]{}; \fill
                    [red] (0.5,0) circle(4pt) node[]{}; \fill [red]
                    (1.5,0) circle(4pt) node[]{}; \fill [blue] (1,0)
                    circle(4pt) node[]{}; \fill [blue] (2,0)
                    circle(4pt) node[]{};
                  \end{tikzpicture}
                                 &
                                   \begin{bmatrix}
                                     0 \\ 1 \\ 1
                                   \end{bmatrix} \\
      - + + 
                & 
                  \begin{tikzpicture}[scale=0.5]
                    \draw (-1, 0) -- (3, 0);
                    \fill [red] (-0.5,0) circle(4pt) node[]{}; \fill
                    [red] (1.5,0.5) circle(4pt) node[]{}; \fill [red]
                    (1.5,0) circle(4pt) node[]{}; \fill [blue] (1,0)
                    circle(4pt) node[]{}; \fill [blue] (2,0)
                    circle(4pt) node[]{};
                  \end{tikzpicture}
                                 &
                                   \begin{bmatrix}
                                     1 \\ 0 \\ 1
                                   \end{bmatrix} \\
      \hline
    \end{array} \qquad 
    \begin{array}{|c|c|c|}
      \hline
       d & \text{Picture} & \GEC(F_d, G) \\
      \hline
      + - - 
                & 
                  \begin{tikzpicture}[scale=0.5]
                    \draw (-1, 0) -- (3, 0);
                    \fill [red] (0.5,0) circle(4pt) node[]{}; \fill
                    [red] (0.5,1) circle(4pt) node[]{}; \fill [red]
                    (0.5,0.5) circle(4pt) node[]{}; \fill [blue] (1,0)
                    circle(4pt) node[]{}; \fill [blue] (2,0)
                    circle(4pt) node[]{};
                  \end{tikzpicture}
                                 &
                                   \begin{bmatrix}
                                     0 \\ 0 \\ 2
                                   \end{bmatrix} \\
      + - + 
                & 
                  \begin{tikzpicture}[scale=0.5]
                    \draw (-1, 0) -- (3, 0);
                    \fill [red] (0.5,0) circle(4pt) node[]{}; \fill
                    [red] (0.5,0.5) circle(4pt) node[]{}; \fill [red]
                    (1.5,0) circle(4pt) node[]{}; \fill [blue] (1,0)
                    circle(4pt) node[]{}; \fill [blue] (2,0)
                    circle(4pt) node[]{};
                  \end{tikzpicture}
                                 &
                                   \begin{bmatrix}
                                     0 \\ 1 \\ 1
                                   \end{bmatrix} \\
      + + - 
                & 
                  \begin{tikzpicture}[scale=0.5]
                    \draw (-1, 0) -- (3, 0);
                    \fill [red] (0.5,0) circle(4pt) node[]{}; \fill
                    [red] (1.5,0) circle(4pt) node[]{}; \fill [red]
                    (0.5,0.5) circle(4pt) node[]{}; \fill [blue] (1,0)
                    circle(4pt) node[]{}; \fill [blue] (2,0)
                    circle(4pt) node[]{};
                  \end{tikzpicture}
                                 &
                                   \begin{bmatrix}
                                     0 \\ 1 \\ 1
                                   \end{bmatrix} \\
      + + + 
                & 
                  \begin{tikzpicture}[scale=0.5]
                    \draw (-1, 0) -- (3, 0);
                    \fill [red] (0.5,0) circle(4pt) node[]{}; \fill
                    [red] (1.5,0.5) circle(4pt) node[]{}; \fill [red]
                    (1.5,0) circle(4pt) node[]{}; \fill [blue] (1,0)
                    circle(4pt) node[]{}; \fill [blue] (2,0)
                    circle(4pt) node[]{};
                  \end{tikzpicture}
                                 &
                                   \begin{bmatrix}
                                     1 \\ 0 \\ 1
                                   \end{bmatrix} \\ \hline
    \end{array}
    \]
  Then, applying Definition \ref{def:EC}, we have
     \[ 
     \EC(F,G) \;\; =\;\; \frac{1}{2^m} \sum_{d \in \{-1, 1\}^m} \GEC(F_d, G) 
            \;\; =\;\; \frac{1}{2^3} \left(
               2
               \begin{bmatrix}
                 0 \\ 0 \\ 2
               \end{bmatrix}
               +
               4
               \begin{bmatrix}
                 0 \\ 1 \\ 1
               \end{bmatrix}
               +
               2
               \begin{bmatrix}
                 1 \\ 0 \\ 1
               \end{bmatrix}
               \right) 
            \;\;=\;\; \frac{1}{8}
               \begin{bmatrix}
                 2 \\ 4 \\ 10 
               \end{bmatrix} 
            \;\;=\;\;
               \begin{bmatrix}
                 1/4 \\ 1/2 \\ 5/4.
               \end{bmatrix}.
   \]
\end{example}

\subsection{Stating the problem}
\label{sec:problem}
In this section, we state the problem precisely.
The goal of this paper is to develop an algorithm for the following problem.

\begin{problem}
\ 
\begin{enumerate}
\item[In:\ ] $F \in \mathbb{R}[p]^{m \times m}$ and  $G \in \mathbb{R}[p]^{n \times n}$, symmetric and generic matrices
where $p$ is a finite set of parameters.

$c \hspace{0.35em}\in\mathbb{N}^{m}$, an eigenvalue configuration.

\item[Out:] a \textquotedblleft simple\ condition\textquotedblright\ on 
$p$ such that $c=\EC\left(F,G\right)  $.
\end{enumerate}

\end{problem}

\begin{remark}
  \label{remark:implicit}
  The problem is essentially a quantifier elimination problem; the input is a condition on (1) the eigenvalues, and (2) the parameters $p$; a condition which, when written in terms of the entries, involves quantifiers. The output is a condition on only $p$ which is quantifier-free condition.
  For a more detailed explanation, see Problem 8 in \cite{signature}.
\end{remark}
\section{Main Result}

\label{sec:main}

In this section, we will state the main theorem. For this, we first introduce two notions which are central to the main result: one is purely \textit{combinatorial} (depending only on the size of $F$) and the other is \textit{algebraic} (depending on the entries/parameters of both $F$ and $G$).
\begin{definition}[Combinatorial part] \label{def:Cy}
  The matrix $\Cy \in \mathbb{Q}^{m \times m}$ is the matrix

  \[ \Cy = V_m W_m \]
  where
\[
  (V_{m})_{rt}=
  \frac{(-1)^{r-1}}{2^{t-1}} \binom{t}{r}
\ \ \ \,\text{and \ }\
(W_m)_{ts} = (-1)^{t+s} \binom{m-s}{m-t}.
\]
  
\end{definition}
\begin{example}
  \label{ex:Cy}
  Let $m = 4$. We will construct the matrix $\Cy$.
  Consider for example the entry at row 3, column 3.
  Then we have
  \begin{alignat*}{2}
    (\Cy)_{3,3} &= \sum_{t=1}^{m} (V_4)_{3, t} (W_4)_{t,3} \\
                &= \sum_{t=1}^{4}
                  \frac{1}{2^{t-1}} \binom{t}{3} (-1)^{t + 3} \binom{1}{4-t} && \text{by definition of $V$ and $W$}\\
                &= 
                  \frac{1}{2^{0}} \binom{1}{3} (-1)^{4} \binom{1}{3} +
                  \frac{1}{2^{1}} \binom{2}{3} (-1)^{5} \binom{1}{2} \\ &\qquad+
                  \frac{1}{2^{2}} \binom{3}{3} (-1)^{6} \binom{1}{1} +
                  \frac{1}{2^{3}} \binom{4}{3} (-1)^{7} \binom{1}{0} \quad &&\text{by expanding the sum}\\
                &= 0 + 0 - \frac{1}{4} + 0 \\
                &= -\frac{1}{4}.
  \end{alignat*}

  Repeating this process for all entries of $\Cy$, we have
\[
  \Cy =  
\left[\begin{array}{cccc}
-\frac{1}{4} & 0 & \frac{1}{4} & \frac{1}{2} 
\\
 0 & \frac{1}{4} & 0 & -\frac{3}{4} 
\\
 \frac{1}{4} & 0 & -\frac{1}{4} & \frac{1}{2} 
\\
 \frac{1}{8} & -\frac{1}{8} & \frac{1}{8} & -\frac{1}{8} 
\end{array}\right].
\]
\exend
\end{example}

\begin{definition}[Algebraic part] \label{def:A}
  The column vector $\Ay \in \mathbb{Z}^m$ is the vector whose $r$-th entry is
  $$(\Ay)_r= \overline{v}(D_r),$$
  where
  \begin{align*}
    \overline{v}(D_r) := v(D_r) + \frac{\tdeg D_r}{2},
  \end{align*}
  where $v$ denotes the sign variation count of a polynomial, and
$D_r$ is
  a polynomial in $\mathbb{R}[a,b][x]$ such that
\[D_r(a, b, x) = h_r(\alpha, \beta, x),\]
where $h_r$ is the polynomial
\begin{align*}
  h_{r}&= \prod_{\substack{I \subset [m], \,\,\#I = r \\ j \in [n] }}
         \left(
          x+\prod_{i \in I}\left(  \alpha_{i}-\beta_{j}\right)  \right),
\end{align*}
  and where $a = (a_1, \dots, a_m)$ and $b = (b_1, \dots, b_n)$ are such that
\begin{align*}
f = \det(xI_{m} - F)  &  = x^{m} - a_{1}x^{m-1} + a_{2} x^{m-2} - \cdots+ (-1)^{m}
a_{m} x^{0}\\
g = \det(xI_{n} - G)  &  = x^{n} - b_{1}x^{n-1} + b_{2} x^{n-2} - \cdots+ (-1)^{n}
b_{n} x^{0}.
\end{align*}
Note the alternating signs and reverse indexing from the usual indexing of
polynomial coefficients.
Equivalently, the coefficient $a_i$ is the coefficient of $x^{m-i}$ in the polynomial $\det(xI_m + F)$ (and similarly with~$b_i$~and~$G$). 
\end{definition}

\begin{example}
  \label{ex:A}
Let $m = n = 2$.
We will construct the polynomial $D_2$ which appears in the definition of the vector $\Ay$.
First, we have that
\begin{align*}
h_{2}= &  \prod_{\substack{I \subset \{1,2\}, \,\, \#I = 2 \\ j \in \{1,2\}}}\left(
x+\prod_{p=1}^{2}\left(  \alpha_{i_{p}}-\beta_{j}\right)  \right)  \\
& (x+(\alpha_{1}-\beta_{1})(\alpha_{2}-\beta_{1}))(x+(\alpha_{1}-\beta
_{2})(\alpha_{2}-\beta_{2})).
\end{align*}
Then, we set
\begin{align*}
D_{2}   
&  =x^{2}+(-a_{1}b_{1}+b_{1}^{2}+2a_{2}-2b_{2})x+a_{1}^{2}b_{2}-a_{1}%
a_{2}b_{1}-a_{1}b_{1}b_{2}+a_{2}b_{1}^{2}+a_{2}^{2}-2a_{2}b_{2}+b_{2}^{2},
\end{align*}
where $a_1, a_2, b_1, b_2$ are the respective coefficients of the characteristic polynomials of $F$ and $G$.
One can then verify that $D_2(a, b, x) = h_2(\alpha, \beta, x)$ by expressing $a_1, a_2, b_1, b_2$ in terms of the eigenvalues~$\alpha$~and~$\beta$.
\exend
\end{example}

\noindent Now we are ready to state our main result.

\begin{theorem}
  [Main Result]\label{thm:main}
  Let $F\in\mathbb{R}^{m\times m}$ and
$G\in\mathbb{R}^{n\times n}$ be arbitrary real symmetric matrices. We have%
\[
      \EC\left(  F,G\right) \,\,\;  = \; \,\,  \Cy\;\;\Ay(F,G).
\]
\end{theorem}
\begin{remark}
  Note that the matrix $\Cy$ is entirely numeric and depends only on $m$.
  In addition, the vector~$\Ay$ depends on the sign variation count of the polynomials $D_r$, whose coefficients are polynomials in the entries of $F$ and $G$.
  Hence, the right-hand side of the above contains no references to the eigenvalues of $F$ and $G$ and is therefore quantifier-free.
\end{remark}
\begin{remark}
  In case the reader is familiar with our related work \cite{signature}, one might notice the similarity between the main theorem of that paper and Theorem \ref{thm:generic} in this paper.
  However, the results are based on completely different ideas; our theorem in \cite{signature} is based on the signature of matrices while the theorem in the current paper is based on real root counting of symmetric polynomials, and to our knowledge, there is no obvious connection.
  We have stated the theorems in a similar way to highlight the superficial similarities between the results: both involve a combinatorial part $C$ which involves only $m$, and an algebraic part $A$ which is constructed via the parameters of $F$ and $G$.
\end{remark}

\begin{example}
  Let $m = 4$ and $n=2$ and let $F$ and $G$ be parametric matrices with each entry being an independent parameter; that is, let
  \[
    F =
    \begin{bmatrix}
      a_{1,1} & a_{1,2} & a_{1,3} & a_{1,4} \\
      a_{1,2} & a_{2,2} & a_{2,3} & a_{2,4} \\
      a_{1,3} & a_{2,3} & a_{3,3} & a_{3,4} \\
      a_{1,4} & a_{2,4} & a_{3,4} & a_{4,4} \\
    \end{bmatrix},
    \qquad \qquad
    G =
    \begin{bmatrix}
      b_{1,1} & b_{1,2} \\ b_{1,2} & b_{2,2}
    \end{bmatrix},
  \]
  where each $a_{ij}$ and $b_{ij}$ is an independent parameter.
  We will now use Theorem \ref{thm:main} to write a condition on these parameters so that the eigenvalues of $F$ and $G$ are arranged as in the following picture.
\[
\begin{tikzpicture}
  \draw (-1, 0) -- (7, 0);
  \fill [red] (0,0) circle(3pt) node[label=$\alpha_1$]{};
  \fill [red] (2,0) circle(3pt) node[label=$\alpha_2$]{};
  \fill [red] (5,0) circle(3pt) node[label=$\alpha_3$]{};
  \fill [red] (6,0) circle(3pt) node[label=$\alpha_4$]{};
  \fill [blue] (1,0) circle(3pt) node[label=$\beta_1$]{};
  \fill [blue] (3,0) circle(3pt) node[label=$\beta_2$]{};
\end{tikzpicture}
\]
That is, we will find a quantifier-free condition for $\EC(F,G) =
\begin{bmatrix}
  1 \\ 1 \\ 0 \\ 0
\end{bmatrix}.$
By Theorem \ref{thm:generic}, we have
\begin{align*}
  \EC(F,G) =
  \begin{bmatrix}
  1 \\ 1 \\ 0 \\ 0
  \end{bmatrix} \qquad \iff \qquad 
  \begin{bmatrix}
  1 \\ 1 \\ 0 \\ 0
  \end{bmatrix} = \Cy \,\, \Ay (F,G).
\end{align*}
In Example \ref{ex:Cy} we found that
\[
  \Cy = 
\left[\begin{array}{cccc}
-\frac{1}{4} & 0 & \frac{1}{4} & \frac{1}{2} 
\\
 0 & \frac{1}{4} & 0 & -\frac{3}{4} 
\\
 \frac{1}{4} & 0 & -\frac{1}{4} & \frac{1}{2} 
\\
 \frac{1}{8} & -\frac{1}{8} & \frac{1}{8} & -\frac{1}{8} 
\end{array}\right].
\]
From Definition \ref{def:A} we have
\[
  \Ay(F,G) =
  \begin{bmatrix}
    \overline{v}(D_1) \\ \overline{v}(D_2) \\ \overline{v}(D_3) \\ \overline{v}(D_4) 
  \end{bmatrix},
\]
where each $D_r$ can be computed using Definition \ref{def:A} as in Example \ref{ex:A}.
Hence we have that
\begin{align*}
  \EC(F,G) =
  \begin{bmatrix}
  1 \\ 1 \\ 0 \\ 0
  \end{bmatrix} \qquad \iff \qquad 
  \begin{bmatrix}
  1 \\ 1 \\ 0 \\ 0
  \end{bmatrix} = 
\left[\begin{array}{cccc}
-\frac{1}{4} & 0 & \frac{1}{4} & \frac{1}{2} 
\\
 0 & \frac{1}{4} & 0 & -\frac{3}{4} 
\\
 \frac{1}{4} & 0 & -\frac{1}{4} & \frac{1}{2} 
\\
 \frac{1}{8} & -\frac{1}{8} & \frac{1}{8} & -\frac{1}{8} 
\end{array}\right]
  \begin{bmatrix}
    \overline{v}(D_1) \\ \overline{v}(D_2) \\ \overline{v}(D_3) \\ \overline{v}(D_4) 
  \end{bmatrix}.
\end{align*}
Solving the linear system above gives the solution
\[
\begin{bmatrix}
  \overline{v}(D_1)  \\ \overline{v}(D_2) \\ \overline{v}(D_3) \\ \overline{v}(D_4)
\end{bmatrix} =
\begin{bmatrix}
  3 \\ 7 \\ 5 \\ 1
\end{bmatrix}.
\]
We therefore have that
\[
  \EC(F,G) =
  \begin{bmatrix}
    1 \\ 1 \\ 0 \\ 0
  \end{bmatrix} \qquad \iff \qquad
\begin{bmatrix}
  \overline{v}(D_1)  \\ \overline{v}(D_2) \\ \overline{v}(D_3) \\ \overline{v}(D_4)
\end{bmatrix}
  =
\begin{bmatrix}
  3 \\ 7 \\ 5 \\ 1
\end{bmatrix}.
\]
Since each $D_r$ is a polynomial with coefficients which are themselves coefficients in the parameters $a_{ij}$ and~$b_{ij}$, the right-hand side of the above is quantifier-free.
  \exend
\end{example}

\begin{remark}
In \cite{signature}, we also approach the eigenvalue configuration problem with an entirely distinct algorithm using the theory of the signature of matrices. This signature approach is structured analogously to Theorem \ref{thm:main}, having a combinatorial part $C$ and an algebraic part $A$.
As such, it is natural to compare the two approaches. To do this, we will examine the total number of terms appearing in each respective $A$ vector (i.e. the size of the $A$ vector times the degree of each polynomial).

\begin{enumerate}
\item {\sl Symmetric polynomials approach}: 
By Definition \ref{def:A}, the size of $\Ay$ is $m$ and the degree of $(\Ay)_r$ equals $n \binom{m}{r}$.
  Hence, the total number of terms is
  \[
    \sum_{r=1}^{m} n \binom{m}{r} \;\;\;=\;\;\; n \sum_{r=1}^{m} \binom{m}{r} \;\;\;=\;\;\; n (2^m - 1).
    \]
  \item {\sl Signature approach}: The size of $A_{\text{sig}}$ is $2^m$ and each entry is a polynomial of degree $n$. However, the first entry in $A_{\text{sig}}$ is constant, so the total number of terms which actually contain information is
    \[
      n (2^m - 1).
      \]
      See the relevant definitions in \cite{signature} for more details.
\end{enumerate}

\noindent Interestingly, both approaches give the same number $n (2^m - 1)$ of total terms, despite the fact that the approaches use totally different ideas. This raises the open question of whether this number is inherent to the eigenvalue configuration problem or there is some hidden connection between the two approaches.
\end{remark}

\section{Proof / Derivation}



\label{sec:proof}
In this section, we will prove the main result (Theorem \ref{thm:main}). We divide the proof into two parts: first, in Section \ref{sec:generic}, we prove the main result in the case where the pair of matrices is generic; then, in Section \ref{sec:arbitrary}, we fully generalize 
\subsection{Proof for generic pairs of matrices} \label{sec:generic}

In this section, we prove the main result in the case where the pair $F$ and $G$ is generic; that is, where $F$ and~$G$ do not share eigenvalues.
We will prove the following theorem.
\begin{theorem}[Main Result for Generic Pairs of Matrices]
  \label{thm:generic}
  Let $F\in\mathbb{R}^{m\times m}$ and
$G\in\mathbb{R}^{n\times n}$ be a \textbf{generic pair} of real symmetric matrices. We have%
\[
      \EC\left(  F,G\right) \,\,\;  = \; \,\,  \Cy\;\;\Ay(F,G).
\]
\end{theorem}

The proof is structured as follows.
\begin{enumerate}
\item First, in Lemmas \ref{lem:tdef} through \ref{iff}, we establish a bijective combinatorial correspondence between the number of positive roots of the $h$ polynomials from Definition \ref{def:A} and the eigenvalue configuration vector.
\item Then, in Proposition \ref{ftsp}, we use the fact that the polynomials $h_r$ are symmetric in the \emph{eigenvalues}~$\alpha$ and~$\beta$ of $F$ and $G$, respectively, to rewrite them in terms of the \textit{coefficients} of the characteristic polynomials of $F$ and $G$. This step eliminates all references to the eigenvalues and therefore concludes the proof.
\end{enumerate}
We will repeatedly revisit the following running example throughout the proof.
\begin{example}[Running example] \label{ex:r1}
\noindent Let
\[
F =
\begin{bmatrix}
4 & 0 \\
0 & 4 
\end{bmatrix}
\in\mathbb{R}^{2 \times2} \qquad\qquad G =
\begin{bmatrix}
2 & 0 & 0\\
0 & 2 & 0\\
0 & 0 & 8
\end{bmatrix}
\in\mathbb{R}^{3 \times3}.
\]
Then their respective eigenvalues are
\[
\alpha= (4, 4) \qquad\qquad\beta= (2,2,8)
\]
\[
\begin{tikzpicture}
  \draw (0, 0) -- (9, 0);
  \fill [red] (4,0) circle(3pt) node[label=$\alpha_1$]{};
  \fill [red] (4,0.8) circle(3pt) node[label=$\alpha_2$]{};
  \fill [blue] (2,0) circle(3pt) node[label=$\beta_1$]{};
  \fill [blue] (2,0.8) circle(3pt) node[label=$\beta_2$]{};
  \fill [blue] (8,0) circle(3pt) node[label=$\beta_3$]{};
\end{tikzpicture}
\]

\noindent So,
\[
  \EC(F,G) =
  \begin{bmatrix}
    c_{1} \\ c_{2}
  \end{bmatrix}
  =
  \begin{bmatrix}
    0 \\ 1
  \end{bmatrix}
  .
\]
\exend
\end{example}

\begin{lemma}[Transform]\label{lem:tdef}We have
\[
c=\EC\left(  F,G\right)  \ \ \ \ \Longrightarrow\ \ \ \ y=T_{m}c
\]
where%
\begin{alignat*}{2}
  y_r &= \# \text{ positive roots of $h_r$, counting multiplicity}  &&\\
  (T_{m})_{rs}&= \# \{\, I \subset [m] \,:\, \#I = r \,\,\wedge\,\, \#\{i \in I: i \le s\} \text{ is odd} \, \} \\
  h_r &= \prod_{\substack{I \subset [m], \,\,\#I = r \\ j \in [n] }} \left(x+\prod_{p=1}^{r}\left(  \alpha_{i_{p}}-\beta_{j}\right)  \right) && \qquad \text{from Definition \ref{def:A}}.
\end{alignat*}
\end{lemma}

\begin{example}[Running example] \label{ex:r2}
  Recall Example \ref{ex:r1}, where we had $F$ and $G$ such that
  \begin{align*}
    \EC(F,G)=
    \begin{bmatrix}
      0 \\ 1
    \end{bmatrix}.
  \end{align*}
  Using the fact that the respective eigenvalues of $F$ and $G$ are $\alpha = (4,4)$ and $\beta = (2,2,8)$, together with the definition of $h_r$ from Definition \ref{def:A}, we compute 
  \begin{align*}
    h_1 &= 
          \prod_{\substack{I \subset [2], \,\,\#I = 1 \\ j \in [3] }} \left(x+\prod_{p=1}^{1}\left(  \alpha_{i_{p}}-\beta_{j}\right)  \right)  \\
        &= \left(x +\alpha_{1}-\beta_{1}\right) \left(x +\alpha_{2}-\beta_{1}\right) \left(x +\alpha_{1}-\beta_{2}\right) \left(x +\alpha_{2}-\beta_{2}\right) \left(x +\alpha_{1}-\beta_{3}\right) \left(x +\alpha_{2}-\beta_{3}\right) \\
        &= (x+2)^4 (x-4)^2 \\
    h_2 &= \left(x +\left(-\beta_{1}+\alpha_{1}\right) \left(-\beta_{1}+\alpha_{2}\right)\right) \left(x +\left(-\beta_{2}+\alpha_{1}\right) \left(-\beta_{2}+\alpha_{2}\right)\right) \left(x +\left(\alpha_{1}-\beta_{3}\right) \left(\alpha_{2}-\beta_{3}\right)\right) \\
        &= (x+4)^2 (x+16).
  \end{align*}
  We therefore have
  \begin{alignat*}{3}
    y_1 &= \# \text{ positive roots of $h_1$, counting multiplicity} = 2 \\
    y_2 &= \# \text{ positive roots of $h_2$, counting multiplicity} = 0.
  \end{alignat*}
  On the other hand, using the definition of $T_m$ we construct
  \begin{align*}
    T_2 &=
          \begin{bmatrix}
            1 & 2 \\ 1 & 0
          \end{bmatrix}.
  \end{align*}
  Therefore
  \begin{align*}
    T_2 c &=
            \begin{bmatrix}
              1 & 2 \\ 1 & 0
            \end{bmatrix}
            \begin{bmatrix}
              0 \\ 1
            \end{bmatrix} =
            \begin{bmatrix}
              2 \\ 0
            \end{bmatrix} = y.
  \end{align*}
  \exend
\end{example}

\begin{proof}[Proof of Lemma \ref{lem:tdef}]
  Assume that $c=\EC\left(  F,G\right)  $. It suffices to show $y=T_{m}c$.
  Recall that
  \[
    y_r = \# \text{ positive roots of $h_r$, counting multiplicity}.
  \]
  Equivalently, using the definition of $h_r$, we have 
\[
y_{r}=\#\left\{ (i_1, \dots, i_r, j) \in Y_r : \prod
_{p=1}^{r}\left(  \alpha_{i_{p}}-\beta_{j}\right)  <0\right\},
\]
where \[
  Y_r = \{ (i_1, \dots, i_r, j) : 1 \le i_1 < \dots < i_r \le m \, \,\wedge\,\, j \in [n]\}.
  \]
We proceed by repeatedly rewriting the definition of $y_{r}$, with the goal of
expressing it in terms of the eigenvalue configuration vector. We begin with
the definition of $y_{r}$.
\[
y_{r}=\#\left\{  \left(  i_{1},\ldots,i_{r},j\right)  \in Y_{r}:\prod
_{p=1}^{r}\left(  \alpha_{i_{p}}-\beta_{j}\right)  <0\right\}
\]
Note that for each $(i_{1},\dots,i_{r})$ satisfying $1\leq i_{1}<\cdots
<i_{r}\leq m$, we also have $(i_{1},\dots,i_{r},j)\in Y_{r}$ simply by
appending each $j=1,\dots,m$. Hence, we can rewrite this action of counting
over the set $Y_{r}$ as a summation over all such tuples $(i_{1},\dots,i_{r}%
)$. Thus we obtain
\[
y_{r}=\sum_{1\leq i_{1}<\cdots<i_{r}\leq m}\ \#\left\{  j:\prod_{p=1}%
^{r}\left(  \alpha_{i_{p}}-\beta_{j}\right)  <0\right\}
\]
We can then introduce another summation by partitioning the set on the
right-hand side depending on which interval $A_{s}$ each $\beta_{j}$ belongs
to.
\[
y_{r}=\sum_{1\leq s\leq m}\ \sum_{1\leq i_{1}<\cdots<i_{r}\leq m}\ \#\left\{
j:\beta_{j}\in A_{s}\ \ \wedge\ \ \prod_{p=1}^{r}\left(  \alpha_{i_{p}}%
-\beta_{j}\right)  <0\right\}
\]
Then, we eliminate the product symbol by observing that the product
$\prod_{p=1}^{r}(\alpha_{i_{p}}-\beta_{j})$ is negative if and only if there
are an odd number of $p$'s such that $\alpha_{i_{p}}-\beta_{j}<0$, or
equivalently $\alpha_{i_{p}}<\beta_{j}$.%

\begin{align*}
y_{r}  &  =\sum_{1\leq s\leq m}\ \sum_{1\leq i_{1}<\cdots<i_{r}\leq
m}\ \#\left\{  j:\beta_{j}\in A_{s}\ \ \wedge\ \ \#\left\{  p:\alpha_{i_{p}%
}<\beta_{j}\right\}  \text{ is odd}\right\}
\end{align*}
Next, we use the fact that the $\alpha$'s are indexed in ascending order, and
we note that $\alpha_{i_{p}} < \beta_{j}$ if and only if~$i_{p} \le s$, since
$\beta_{j} \in A_{s}$. Thus
\begin{align*}
y_{r}  &  =\sum_{1\leq s\leq m}\ \sum_{1\leq i_{1}<\cdots<i_{s}\leq
m}\ \#\left\{  j:\beta_{j}\in A_{s}\ \ \wedge\ \ \#\left\{  p:i_{p}\leq
s\right\}  \text{ is odd}\right\}
\end{align*}
Now, note that if $\# \{p : i_{p} \le s\}$ is even, then the size of the set
in the summand is zero. If $\# \{p : i_{p} \le s\}$ is odd, then the size of
the set in the summand equals $\#\{j : \beta_{j} \in A_{s}\}$, which is
exactly $c_{s}$. Hence
\begin{align*}
y_{s}  &  =\sum_{1\leq s\leq m}\ \sum_{1\leq i_{1}<\cdots<i_{r}\leq m}\left\{
\begin{array}
[c]{ll}%
0 & \text{if\ \ }\#\left\{  p:i_{p}\leq s\right\}  \text{ is even}\\
c_{s} & \text{if\ \ }\#\left\{  p:i_{p}\leq s\right\}  \text{ is odd}
\end{array}
\right.
\end{align*}
Next, we factor out the $c_{s}$ and move it outside the innermost summation.
\begin{align*}
y_{r}  &  =\sum_{1\leq s\leq m}\ \sum_{1\leq i_{1}<\cdots<i_{r}\leq m}\left(
\left\{
\begin{array}
[c]{ll}%
0 & \text{if\ \ }\#\left\{  p:i_{p}\leq s\right\}  \text{ is even}\\
1 & \text{if\ \ }\#\left\{  p:i_{p}\leq s\right\}  \text{ is odd}
\end{array}
\right.  c_{s}\right) \\
&  =\sum_{1\leq s\leq m}\ \left(  \sum_{1\leq i_{1}<\cdots<i_{r}\leq
m}\left\{
\begin{array}
[c]{ll}%
0 & \text{if\ \ }\#\left\{  p:i_{p}\leq r\right\}  \text{ is even}\\
1 & \text{if\ \ }\#\left\{  p:i_{p}\leq r\right\}  \text{ is odd}
\end{array}
\right.  \right)  c_{s}%
\end{align*}
Then, we fold the conditional into the summation.
\[
y_{r}=\sum_{1\leq s\leq m}\left(  \sum_{\substack{1\leq i_{1}<\cdots<i_{r}\leq
m\\\#\left\{  p:i_{p}\leq s\right\}  \text{ is odd}}}1\right)  c_{s}%
\]
Since the innermost summand is just a summation of 1's, we can view it as
counting the elements of the set~$\{(i_{1},\dots,i_{r}):1\leq i_{1}%
<\cdots<i_{r}\leq m\,\wedge\,\#\{p:i_{p}\leq s\}\text{ is odd}\}$.
Rewriting this, we see that
\begin{align*}
  \#\{(i_{1},\dots,i_{r}):1\leq i_{1}%
  <\cdots<i_{r}\leq m\,\wedge\,\#\{p:i_{p}\leq s\}\text{ is odd}\} &= \#\{ I \subset [m] : \# I = r \, \, \wedge \,\, \#\{i \in I : i \le s\} \text{ is odd} \} \\&= (T_m)_{rs}.
\end{align*}
Hence
\[
y_{r}=\sum_{1\leq s\leq m}(T_{m})_{rs}c_{s}.
\]
By the definition of matrix multiplication, this is just the multiplication of
the vector $c=(c_{1},\dots,c_{m})^{T}$ by the matrix $T_{m}$ whose entries are
defined as above. Thus%
\[
\left(
\begin{array}
[c]{c}%
y_{1}\\
\vdots\\
y_{m}%
\end{array}
\right)  =\left(
\begin{array}
[c]{ccc}%
(T_{m})_{11} & \cdots & (T_{m})_{1m}\\
\vdots &  & \vdots\\
(T_{m})_{m1} & \cdots & (T_{m})_{mm}%
\end{array}
\right)  \left(
\begin{array}
[c]{c}%
c_{1}\\
\vdots\\
c_{m}%
\end{array}
\right)
\]
that is%
\[
y=T_{m}c.
\]
Therefore
\[
c=\EC\left(  F,G\right)  \ \ \ \ \Longrightarrow\ \ \ \ y=T_{m}c.
\]
\end{proof}

Now, we will establish the other direction of the correspondence between $c$ and $y$. To do this, we will give a closed-form expression for $T_m$, and show that it is invertible.

\begin{lemma}
  \label{lem:inv} We have $T_m = L_m U_m$ where
  \begin{align*}
    (L_m)_{rt} &= \binom{m-t}{r-t} \\
    (U_m)_{ts} &= (-2)^{t-1} \binom{s}{t}.
  \end{align*}
  In particular, we have that $T_m$ is an invertible matrix.
\end{lemma}

\begin{example}[Running example] \label{ex:r3}
  In Example \ref{ex:r2}, we found that
  \begin{align*}
    T_2 &=
          \begin{bmatrix}
            1 & 2 \\ 1 & 0
          \end{bmatrix}.
  \end{align*}
  Observe that
  \begin{align*}
    \det(T_2) = 0 - 2 = -2 \ne 0,
  \end{align*}
  and so $T_2$ is invertible.
  \exend
\end{example}

\begin{proof}[Proof of Lemma \ref{lem:inv}]
It suffices to prove that $\det\left( T_{m}\right)\neq0$. We prove it in
several stages. First, we establish a recurrence relation for the entries of
$T_{m}$, and then apply that to decompose $T_{m}$ into a product of triangular
matrices to more easily compute the determinant.

\bigskip

\noindent\textsf{Claim 1:} We have%
\[
(T_{m})_{rs}=\left\{
\begin{array}
[c]{ll}%
s & \text{if }\left(  r,s\right)  \in\left\{  1\right\}  \times\left\{
1,\ldots,m\right\} \\
\noalign{\vspace*{1mm}} \left\{
\begin{array}
[c]{ll}%
0 & \text{if }s\ \text{is even} \\
1 & \text{if }s\ \text{is odd}
\end{array}
\right.  & \text{if }\left(  r,s\right)  \in\left\{  m\right\}  \times\left\{
1,\ldots,m\right\} \\
\noalign{\vspace*{2mm}} \left\{
\begin{array}
[c]{ll}%
\left(
\begin{array}
[c]{c}%
m\\
r
\end{array}
\right)  & \text{if }r\ \text{is odd}\\
0 & \text{if }r\ \text{is even}%
\end{array}
\right.  & \text{if }\left(  r,s\right)  \in\left\{  2,\ldots,m-1\right\}
\times\left\{  m\right\} \\
\noalign{\vspace*{2mm}} (T_{m-1})_{r-1,s}+(T_{m-1})_{r,s} & \text{if }\left(
r,s\right)  \in\left\{  2,\ldots,m-1\right\}  \times\left\{  1,\ldots
,m-1\right.  \}.
\end{array}
\right.
\]
Proof of the claim: There are four cases in the above. We will prove them one
by one.

\begin{enumerate}
\item $\left(  r,s\right)  \in\left\{  1\right\}  \times\left\{
1,\ldots,m\right\}  $.

By definition, we have
\[
  (T_m)_{1s} \,\,=\,\, \# \{\, I \subset [m] \,:\, \#I = 1 \,\,\wedge\,\, \#\{i \in I: i \le s\} \text{ is odd} \, \}.
\]
Since we are only considering $\#I = 1$, then $\#\{i \in I : i\leq s\}$ can only be
zero or one. If that quantity is zero, then $(T_{m})_{1,s}$ is zero, since
zero is even. If instead $\#\{i \in I : i\leq s\}=1$, then the single element of $I$ is between $1$ and $s$, so
there are $s$ choices for that element. Thus, in this case $(T_{m})_{1,s}=s$.
Together, we have that
\[
  (T_m)_{1s} =
  \begin{cases}
    0 & \text{ if $s$ is even.} \\
    1 & \text{ if $s$ is odd}. 
  \end{cases}
\]

\item $\left(  r,s\right)  \in\left\{  m\right\}  \times\left\{
1,\ldots,m\right\}  $.

By definition, we have
\begin{align*}
  (T_m)_{ms} &= \# \{\, I \subset [m] \,:\, \#I = m \,\,\wedge\,\, \#\{i \in I: i \le s\} \text{ is odd} \, \} \\
&  =\left\{
\begin{array}
[c]{ll}%
1 & \text{if }s\ \text{is odd}\\
0 & \text{if }s\ \text{is even}.
\end{array}
\right.
\end{align*}

\item $\left(  r,s\right)  \in\left\{  2,\ldots,m-1\right\}  \times\left\{
m\right\}  $.

By definition of $(T_{m})_{rs}$ we have
\[
  (T_m)_{rm} = \# \{\, I \subset [m] \,:\, \#I = r \,\,\wedge\,\, \#\{i \in I: i \le m\} \text{ is odd} \, \}.
\]
But $i \leq m$ for all $i \in I$ trivially, so $\#\{i \in I : i \le m\}=\#I = r$. Therefore
\begin{align*}
  (T_m)_{rm} &= \# \{\, I \subset [m] \,:\, \#I = r \,\,\wedge\,\, \#\{i \in I: i \le m\} \text{ is odd} \, \} \\
&  =\#\left\{ I \subset [m] : \#I = r  \ \ \ \wedge\ \ \ r\text{ is odd}\right\} \\
&  =\left\{
\begin{array}
[c]{ll}%
\#\left\{  I \subset [m] : \#I = r \right\}  & \text{if }r\ \text{is odd}\\
\#\left\{  {}\right\}  & \text{if }r\ \text{is even}%
\end{array}
\right. \\
&  =\left\{
\begin{array}
[c]{ll}%
\left(
\begin{array}
[c]{c}%
m\\
r
\end{array}
\right)  & \text{if }r\ \text{is odd}\\
0 & \text{if }r\ \text{is even}.
\end{array}
\right.
\end{align*}

\item $\left(  r,s\right)  \in\left\{  2,\ldots,m-1\right\}  \times\left\{
1,\ldots,m-1\right\}  $.

By definition, we have%
\[
  (T_m)_{rs} = \# \{\, I \subset [m] \,:\, \#I = r \,\,\wedge\,\, \#\{i \in I: i \le s\} \text{ is odd} \, \}.
\]
We now partition the set we are counting into disjoint sets with $m \in I$ and $m \not \in I$.
Then we have
\begin{align*}
  (T_m)_{rs} = \qquad &\# \{\, I \subset [m] \,:\, \#I = r \,\,\wedge\,\, m \in I \,\,\wedge\,\, \#\{i \in I: i \le s\} \text{ is odd} \, \} \\
  + &\# \{\, I \subset [m] \,:\, \#I = r \,\,\wedge\,\, m \not \in I \,\,\wedge\,\, \#\{i \in I: i \le s\} \text{ is odd} \, \}.
\end{align*}
If $m \in I$, then $I \setminus \{m\} \subset [m-1]$. In addition, we still have that $\#\{i \in I: i\leq
s\}$ is odd, because $s<m$. These are exactly the tuples that are counted by
$(T_{m-1})_{r-1,s}$; thus
\begin{align*}
\# \{\, I \subset [m] \,:\, \#I = r \,\,\wedge\,\, m \in I \,\,\wedge\,\, \#\{i \in I: i \le s\} \text{ is odd} \, \}
\,\,=\,\,    (T_{m-1})_{r-1,s}.
\end{align*}
On the other hand, if $m \not \in I$, then $I \subset [m-1]$ and again $\#\{i \in I: i \leq s\}$
remains odd. These tuples are counted by $(T_{m})_{r,s}$, and so
\begin{align*}
  \# \{\, I \subset [m] \,:\, \#I = r \,\,\wedge\,\, m \not \in I \,\,\wedge\,\, \#\{i \in I: i \le s\} \text{ is odd} \, \} \,\,
=  \,\,  (T_{m-1})_{r,s}.
\end{align*}
Putting those together, we get that
\[
(T_{m})_{rs}=(T_{m-1})_{r-1,s}+(T_{m-1})_{r,s}.
\]

\end{enumerate}

\noindent We have proved Claim 1.

\bigskip

\noindent\textsf{Claim 2:} We have
\[
T_{m}=L_{m}U_{m}%
\]
where
\[
(L_{m})_{rt}=\left(
\begin{array}
[c]{c}%
m-t\\
r-t
\end{array}
\right)  \ \ \ \,\text{and \ }\ (U_{m})_{ts}=\left(  -2\right)  ^{t-1}\left(
\begin{array}
[c]{c}%
s\\
t
\end{array}
\right);
\]
that is, for $1 \le r,s \le m$ we have
$$(T_m)_{rs} = \sum_{t=1}^{m} \binom{m-t}{r-t} (-2)^{t-1}
\binom{s}{t}.$$

\noindent Proof of the claim. It suffices to show%
\[
(T_{m})_{rs}=\sum_{t=1}^{m}(L_{m})_{rt}(U_{m})_{ts}%
\]
We proceed by induction on $m$. Note that when $m=1$, we have
\[
1=(T_{1})_{11}=\binom{1-1}{1-1}(-2)^{1-1}\binom{1}{1}=L_{1}U_{1}.
\]
For the induction hypothesis, suppose the claim is true for some $m-1$. We
will prove the claim for $m$. There are four cases in the above. We will prove
them one by one.

\begin{enumerate}
\item $\left(  r,s\right)  \in\left\{  1\right\}  \times\left\{
1,\ldots,m\right\}  $.

We have
\begin{align*}
&  \sum_{t=1}^{m}(L_{m})_{1t}(U_{m})_{ts}\\
&  =\sum_{t=1}^{m}\binom{m-t}{1-t}(-2)^{t-1}\binom{s}{t}\\
&  =\binom{m-1}{1-1}(-2)^{1-1}\binom{s}{1}\ \ \ \ \text{since }\binom
{m-t}{1-t}=0\ \text{for }t>1\\
&  =s\\
&  =(T_{m})_{1s}%
\end{align*}

\item $\left(  r,s\right)  \in\left\{  m\right\}  \times\left\{
1,\ldots,m\right\}  $.

We have
\begin{align*}
&  \sum_{t=1}^{m}(L_{m})_{mt}(U_{m})_{ts}\\
&  =\sum_{t=1}^{m}\left(
\begin{array}
[c]{c}%
m-t\\
m-t
\end{array}
\right)  \left(  -2\right)  ^{t-1}\left(
\begin{array}
[c]{c}%
s\\
t
\end{array}
\right) \\
&  =\sum_{t=1}^{s}\left(  -2\right)  ^{t-1}\left(
\begin{array}
[c]{c}%
s\\
t
\end{array}
\right)  \ \ \,\,\,\text{since }\left(
\begin{array}
[c]{c}%
s\\
t
\end{array}
\right)  =0\ \,\text{for }t>s\\
&  =\frac{1}{-2}\left(  -1+\sum_{t=0}^{s}\left(  1\right)  ^{s-t}\left(
-2\right)  ^{t}\left(
\begin{array}
[c]{c}%
s\\
t
\end{array}
\right)  \right) \\
&  =\frac{1}{-2}\left(  -1+\left(  1-2\right)  ^{s}\right) \\
&  =\left\{
\begin{array}
[c]{ll}%
1 & \text{if }s\ \text{is odd}\\
0 & \text{if }s\ \text{is even}%
\end{array}
\right. \\
&  =\left(  T_{m}\right)  _{ms}%
\end{align*}

\item $\left(  r,s\right)  \in\left\{  2,\ldots,m-1\right\}  \times\left\{
m\right\}  $.

We have
\begin{align*}
\sum_{t=1}^{m}(L_{m})_{rt}(U_{m})_{tm}  &  =\sum_{t=1}^{n}\binom{m-t}%
{r-t}(-2)^{t-1}\binom{m}{t}\\
&  =\sum_{t=1}^{m}\frac{\left(  m-t\right)  !}{\left(  m-r\right)  !\left(
r-t\right)  !}(-2)^{t-1}\frac{m!}{\left(  m-t\right)  !t!}\\
&  =\sum_{t=1}^{m}\frac{m!}{\left(  m-r\right)  !}\frac{1}{\left(  r-t\right)
!t!}(-2)^{t-1}\\
&  =\sum_{t=1}^{m}\frac{m!}{\left(  m-r\right)  !r!}\frac{r!}{\left(
r-t\right)  !t!}(-2)^{t-1}\\
&  =\sum_{t=1}^{m}\left(
\begin{array}
[c]{c}%
m\\
r
\end{array}
\right)  \left(
\begin{array}
[c]{c}%
r\\
t
\end{array}
\right)  (-2)^{t-1}\\
&  =\left(
\begin{array}
[c]{c}%
m\\
r
\end{array}
\right)  \sum_{t=1}^{m}\left(
\begin{array}
[c]{c}%
r\\
t
\end{array}
\right)  (-2)^{t-1}\\
&  =\left(
\begin{array}
[c]{c}%
m\\
r
\end{array}
\right)  \left(  \frac{1}{-2}\left(  -1+\sum_{t=0}^{m}\left(
\begin{array}
[c]{c}%
r\\
t
\end{array}
\right)  \left(  1\right)  ^{r-t}(-2)^{t}\right)  \right) \\
&  =\left(
\begin{array}
[c]{c}%
m\\
r
\end{array}
\right)  \left(  \frac{1}{-2}\left(  -1+\left(  1-2\right)  ^{r}\right)
\right) \\
&  =\left\{
\begin{array}
[c]{ll}%
\left(
\begin{array}
[c]{c}%
m\\
r
\end{array}
\right)  & \text{if }r\ \text{is odd}\\
0 & \text{if }r\ \text{is even}%
\end{array}
\right. \\
&  =\left(  T_{m}\right)  _{rm}%
\end{align*}

\item $\left(  r,s\right)  \in\left\{  2,\ldots,m-1\right\}  \times\left\{
1,\ldots,m-1\right\}  $.

Recall that we have
\[
(T_{m})_{rs}=(T_{m-1})_{r-1,s}+(T_{m-1})_{r,s}.
\]
By the induction hypothesis, we have
\begin{align*}
(T_{m})_{rs}  &  =\sum_{t=1}^{m-1}(L_{m-1})_{r-1,t}(U_{m-1})_{t,s}+\sum
_{t=1}^{m-1}(L_{m-1})_{r,t}(U_{m-1})_{t,s}\\
&  =\sum_{t=1}^{m-1}(L_{m-1})_{r-1,t}(U_{m-1})_{t,s}+(L_{m-1})_{r,t}%
(U_{m-1})_{t,s}\\
&  =\sum_{t=1}^{m-1}((L_{m-1})_{r-1,t}+(L_{m-1})_{r,t})(U_{m-1})_{t,s}\\
&  =\sum_{t=1}^{m-1}\left(  \binom{m-1-t}{r-1-t}+\binom{m-1-t}{r-t}\right)
(U_{m-1})_{t,s}\\
&  =\sum_{t=1}^{m-1}\binom{m-t}{r-t}(U_{m-1})_{t,s}\\
&  =\sum_{t=1}^{m-1}(L_{m})_{r,t}(U_{m-1})_{t,s}.
\end{align*}
But $(U_{m-1})_{t,s}$ does not depend on $m$. Thus $(U_{m-1})_{t,s}%
=(U_{m})_{t,s}$. Hence
\begin{align*}
(T_{m})_{rs}  &  =\sum_{t=1}^{m-1}(L_{m})_{r,t}(U_{m})_{t,s}\\
&  =\sum_{t=1}^{m-1}(L_{m})_{r,t}(U_{m})_{t,s}+\underbrace{\binom{m-m}%
{r-m}(-2)^{r-1}\binom{s}{m}}_{=0\text{ since }s<m}\\
&  =\sum_{t=1}^{m}(L_{m})_{r,t}(U_{m})_{t,s}.
\end{align*}

\end{enumerate}

\noindent We have proved Claim 2.

\bigskip

\noindent Now note that $L_{m}$ is lower triangular, since if $s>r$ then
$r-s<0$ and so $\binom{m-s}{r-s}=0$. Similarly, $U_{m}$ is upper triangular,
since if $r>s$ then $\binom{s}{r}=0$. Further, we have
\[
\det(L_{m})=\prod_{t=1}^{m}(L_{m})_{tt}=\prod_{t=1}^{m}\binom{m-t}{t-t}%
=\prod_{t=1}^{m}\binom{m-t}{0}=\prod_{t=1}^{m}1=1,
\]
and
\[
\det(U_{m})=\prod_{t=1}^{m}(U_{m})_{tt}=\prod_{t=1}^{m}(-2)^{t-1}\binom{t}%
{t}=\prod_{t=1}^{m}(-2)^{t-1}=(-2)^{0+1+\dots+m-1}=(-2)^{\binom{m}{2}}.
\]
Thus we have
\[
\det(T_{m})=\det(L_{m})\det(U_{m})=(-2)^{\binom{m}{2}}\neq0\text{.}%
\]
Hence $T_{m}$ is invertible.
\end{proof}


Next, we establish the relationship between $\Cy$ and the matrix $T_m$; in particular, they are inverses of each other.
\begin{proposition} \label{prop:CT}
  We have that $\Cy = T_m^{-1}$.
\end{proposition}
\begin{proof}
  Recall from Defintion \ref{def:Cy} that $\Cy = V_mW_m$ where
\[
  (V_{m})_{rt}=
  \frac{(-1)^{r-1}}{2^{t-1}} \binom{t}{r}
\ \ \ \,\text{and \ }\
(W_m)_{ts} = (-1)^{t+s} \binom{m-s}{m-t}.
\]
  Recall again from Lemma \ref{lem:inv} that $T_m = L_m U_m$ where
\[
(L_{m})_{rt}=\left(
\begin{array}
[c]{c}%
m-t\\
r-t
\end{array}
\right)  \ \ \ \,\text{and \ }\ (U_{m})_{ts}=\left(  -2\right)  ^{t-1}\left(
\begin{array}
[c]{c}%
s\\
t
\end{array}
\right).
\]
Note that
\begin{align*}
  T_m^{-1} = (L_m U_m)^{-1} = U_m^{-1} L_m^{-1};
\end{align*}
hence, to show that $\Cy = T_m^{-1}$, it suffices to show that
\begin{enumerate}
\item $V_mU_m = I_m$
\item $L_m W_m = I_m$
\end{enumerate}
where $I_m$ denotes the $m \times m$ identity matrix.
\begin{enumerate}
\item$V_mU_m = I_m$:
  We have that the $(r,s)$-entry of $V_mU_m$ is
  \begin{alignat*}{2}
    (V_mU_m)_{rs} &= \sum_{t}^{} V_{rt} U_{ts} \\
                  &= \sum_{t}^{} \frac{(-1)^{r-1}}{2^{t-1}} \binom{t}{r} (-2)^{t-1} \binom{s}{t} \\
                  &= \sum_{t}^{} (-1)^{r + t} \binom{t}{r} \binom{s}{t} \\
                  &= \sum_{u}^{} (-1)^{u} \binom{u+r}{r} \binom{s}{u + r}&&  \qquad \text{by reindexing with $t = u + r$} \\
                  &= \sum_{u}^{} (-1)^u \frac{(u+r)!}{r!u!} \binom{s!}{(u+r)! (s-u-r)!} \\
                  &= \sum_{u}^{} (-1)^u \frac{s!}{r!u!(s-u-r)!} \cdot \frac{(s-r)!}{(s-r)!} &&  \qquad \text{by multiplying by $\frac{(s-r)!}{(s-r)!}$} \\
                  &= \sum_{u}^{} (-1)^u \frac{s!}{r! (s-r)!} \frac{(s-r)!}{u!(s-r-u)!} \\
                  &= \sum_{u}^{} (-1)^u \binom{s}{r} \binom{s-r}{u} \\
                  &= \binom{s}{r} \sum_{u}^{} (-1)^u \binom{s-r}{u} (1)^{s-r-u} \\
                  &= \binom{s}{r} (1-1)^{s-r} &&  \qquad \text{binomial expansion of $(1-1)^{s-r}$} \\
                  &=
                    \begin{cases}
                      1 & \text{if $s = r$} \\
                      0 & \text{if $s \ne r$}
                    \end{cases}  \\
                  &= \delta_{r,s}.
  \end{alignat*}
  Since the $(r,s)$-entry of $V_mU_m$ is $\delta_{r,s}$, it follows that $V_mU_m = I_m$.
\item $L_mW_m = I_m$:
  We have that the $(r,s)$-entry of $L_mW_m$ is
  \begin{alignat*}{2}
    (L_mW_m)_{rs} &= \sum_{t}^{} (L_m)_{rt} (W_m)_{ts} \\
                  &= \sum_{t}^{} \binom{m-t}{r-t} (-1)^{t + s} \binom{m-s}{m-t} \\
                  &= \sum_{u}^{} \binom{m-r+u}{u} (-1)^{r - u + s} \binom{m-s}{m-r+u} && \qquad \text{by reindexing with $t = r - u$} \\
                  &= \sum_{u}^{} (-1)^{r - u + s} \frac{(m-r+u)!}{u! (m-r)!} \binom{(m-s)!}{(m-r+u)! (r-s-u)!} \\
                  &= \sum_{u}^{} (-1)^{r - u + s} \binom{(m-s)!}{u! (m-r)! (r-s-u)!} \cdot \frac{(r-s)!}{(r-s)!} && \qquad \text{by multiplying by $ \frac{(r-s)!}{(r-s)!}$} \\
                  &= \sum_{u}^{} (-1)^{r - u + s} \frac{(r-s)!}{u! (r-s-u)!} \frac{(m-s)!}{(m-r)!(r-s)!} \\
                  &= \sum_{u}^{} (-1)^{r - u + s} \binom{r-s}{u} \binom{m-s}{m-r} \\
                  &= (-1)^{r + s} \binom{m-s}{m-r} \sum_{u}^{} (-1)^u \binom{r-s}{u} (1)^{r - s - u} \\
                  &= (-1)^{r + s} \binom{m-s}{m-r} (1-1)^{r - s} && \qquad \text{binomial expansion of $(1-1)^{r-s}$} \\
                  &=
                    \begin{cases}
                      1 & \text{if $r = s$} \\
                      0 & \text{if $r \ne s$}
                    \end{cases} \\
                  &= \delta_{r,s}.
  \end{alignat*}
  Since the $(r,s)$-entry of $L_mW_m$ is $\delta_{r,s}$, it follows that $L_mW_m = I_m$.
\end{enumerate}
Hence, we have shown that $V_mU_m = I_m$ and $L_mW_m = I_m$, and therefore
\[
  \Cy = V_m W_m
      = U_m^{-1} L_m^{-1}
      = (L_m U_m)^{-1}
      = T_m^{-1}.
\]
\end{proof}

\noindent
Now, we can summarize Lemmas \ref{lem:tdef} and \ref{lem:inv} and Proposition \ref{prop:CT} into the following lemma.

\begin{lemma}
\label{iff} We have
\[
\EC(F, G)  \ \ = \ \ \Cy \,\, y.
\]

\end{lemma}

\begin{proof}
Lemma
\ref{lem:inv} shows that $T_{m}$ is invertible, hence~$T_{m}$ is a 1-1 linear map.
Proposition \ref{prop:CT} shows that $\Cy = T_m^{-1}$.
Therefore, we have that$${T_m^{-1}y = \Cy \,\, y = c \implies \EC(F, G) = c.}$$
Together with Lemma \ref{lem:tdef}, this means that $\EC(F,G) = T_m^{-1} y = \Cy \,\, y$, and the lemma is proved.
\end{proof}

At this point, we have found an equivalent condition for $\EC(F,G)=c$. However,
this condition still contains quantifiers, as the vector $y$ (via the $h$ polynomials) is still computed using the eigenvalues of~$F$~and~$G$. The remainder of the derivation will focus on solving this issue by providing a way to count the positive roots of the $h$ polynomials without referring to the eigenvalues.




Recall that, for each $r \in [m]$, the polynomial $h_{r}$ is symmetric with only real roots. This is the
crucial fact which allows us to construct another polynomial $D_{r}$ involving
only the \textit{parameters} of $F$ and $G$ rather than their \textit{eigenvalues}.
This will give us a quantifier-free condition.

\begin{proposition}
\label{ftsp} Let $\alpha = (\alpha_1, \dots, \alpha_m)$, and similarly for $\beta$, $a$, and $b$. For each $r \in [m]$, there exists a polynomial~${D_{r}\in\mathbb{R}[a,b][x]}$
 such that
\[
  D_r(a,b,x) \,\, = \,\, h_r(\alpha, \beta, x).
\]

\end{proposition}

\begin{proof}
  Let $r \in [m]$ be arbitrary but fixed.
  \ For the purposes of this proof, we view $h_{r}$ as being a polynomial in the variables
  \begin{align*}
    \alpha &= (\alpha_{1}, \dots, \alpha_{m}), \\
    \beta &= ( \beta_{1}, \dots, \beta_{n}), \\
    x &,
  \end{align*}
  with real number coefficients.
  Recall from Definition \ref{def:A} that
  \[
    h_r = 
   \prod_{\substack{I \subset [m], \,\,\#I = r \\ j \in [n] }}
         \left(
          x+\prod_{i \in I}\left(  \alpha_{i}-\beta_{j}\right)  \right).
  \]
  Note that by construction $h_{r}$ is symmetric in both $\alpha$ and $\beta$.
  We will now apply the Fundamental Theorem of Symmetric Polynomials.
  For this, let $e_k(\alpha)$ denote the $k$-th elementary symmetric polynomial in the variables $\alpha = (\alpha_1, \dots, \alpha_m)$; that is,
  \[
    e_k(\alpha) = \sum_{1 \le i_1 < \dots < i_k \le m} \alpha_{i_1} \cdots \alpha_{i_k}.
  \]
  Let $e(\alpha)$ denote the list of all elementary symmetric polynomials; i.e.,
  \[
    e(\alpha) = (e_1(\alpha), \dots, e_m(\alpha)),
  \]
  and respectively for $e_k(\beta)$ and $e(\beta)$.

  Recall that by the Fundamental Theorem of Symmetric Polynomials, since $h_r$ is symmetric in $\alpha$ and $\beta$ this means that there exists a polynomial 
$$D_r\in \mathbb{Z}[y_{1},\dots,y_{m}, \,\, z_1, \dots, z_n, \,\, x]$$
such that
\[
D_r(e(\alpha), \,\, e(\beta), \,\, x)=h_{r}(\alpha, \, \beta, \, x).
\]
Recall from Definition \ref{def:A} that $a = (a_1, \dots, a_m)$ and $b = (b_1, \dots, b_m)$ were the coefficients of the respective characteristic polynomials $f$ and $g$ of $F$ and $G$ labelled so that
\begin{alignat*}{2}
f &= \det(xI_{m} - F)  && = x^{m} - a_{1}x^{m-1} + a_{2} x^{m-2} - \cdots+ (-1)^{m}
a_{m} x^{0}\\
g &= \det(xI_{n} - G)  &&  = x^{n} - b_{1}x^{n-1} + b_{2} x^{n-2} - \cdots+ (-1)^{n}
b_{n} x^{0}.
\end{alignat*}
With this labelling of the coefficients of $f$ and $g$, we have that
\begin{align*}
a_{1}  &  =\sum_{1\leq i_{1}\leq m}\alpha_{i_{1}}=\alpha_{1}+\cdots+\alpha
_{m}=e_{1}(\alpha)\\
a_{2}  &  =\sum_{1\leq i_{1}<i_{2}\leq m}\alpha_{i_{1}}\alpha_{i_{2}}%
=e_{2}(\alpha)\\
&  \vdots\\
a_{m}  &  =\sum_{1\leq i_{1}<\cdots<i_{m}\leq m}\alpha_{i_{1}}\cdots
\alpha_{i_{m}}=\alpha_{1}\cdots\alpha_{m}=e_{m}(\alpha).
\end{align*}
Similarly
\begin{align*}
b_{1}  &  =e_{1}(\beta)\\
&  \vdots\\
b_{n}  &  =e_{n}(\beta).
\end{align*}

\noindent Hence, we have that
\begin{align*}
h_{r}(\alpha,\beta,x)  &  = D_{r}(
\underbrace{e_{1}(\alpha)}_{a_{1}}, \dots,  \underbrace{e_{m}%
(\alpha)}_{a_{m}}, \,\, \underbrace{e_{1}(\beta)}_{b_{1}},  \dots, 
\underbrace{e_{n}(\beta)}_{b_{n}}, \,\, x)\\
&  =D_{r}(a,b,x) \in \mathbb{R}[a,b][x].
\end{align*}

\end{proof}

\noindent With that, we are ready to prove the result for this section (Theorem \ref{thm:generic}).

\begin{proof}
[\textbf{Proof of Main Result for Generic Pairs (Theorem \ref{thm:generic})}] \ 

Let $y=%
\begin{bmatrix}
y_{1}\\
\vdots\\
y_{m}%
\end{bmatrix}
=T_{m}c$. For each $r\in [m]$, by Proposition \ref{ftsp} there exists
$D_{r}\in\mathbb{R}[a,b][x]$ such that
\[
D_r(a, b, x) = h_{r}(\alpha,\beta,x).
\]
Recall that, by definition of $y$, we have
\begin{alignat*}{2}
y_{r}  &  =\#\text{ positive roots of $h_{r}(x)$, counting multiplicity}\\
       &  =\#\text{ positive roots of $D_r(x)$, counting multiplicity} \qquad && \text{ by Proposition \ref{ftsp}} \\
       &  =v(D_{r}(x)) && \text{ by Descartes' rule of signs, since all roots of $D_r$ are real.}
\end{alignat*}
Next, note that by construction of $D_r$, we have that $\tdeg D_r = 0$ since $F$ and $G$ do not share any eigenvalues.
Hence
\[
  y_r \;\;=\;\; v(D_r(x)) \\
      \;\;=\;\; v(D_r(x)) + \frac{\tdeg D_r}{2} \\
      \;\;=\;\; (\Ay(F,G))_r.
\]
By the above and Lemma \ref{iff}, we have that
\[
\EC(F,G)  \;\;=\;\; T_{m}^{-1} y\\
\;\; =\;\; T_{m}^{-1}
\begin{bmatrix}
  v(D_1) \\ \vdots \\ v(D_m)
\end{bmatrix} \\
                \;\;=\;\; \Cy \,\, \Ay (F,G).
\]
Thus
Theorem \ref{thm:generic} is proved.
\end{proof}

\subsection{Proof for arbitrary pairs of matrices}
\label{sec:arbitrary}
In this section, we will generalize Theorem~\ref{thm:generic} from the previous section to arbitrary pairs of matrices. 
Namely, we will prove the main result (Theorem \ref{thm:main}) for this paper (restated below).

\begin{theorem*}[Main Result]\label{thm:arbitrary} Let $F\in\mathbb{R}^{m\times m}$ and
$G\in\mathbb{R}^{n\times n}$ be \textbf{arbitrary} real symmetric matrices. We have%
\[
  \EC\left(  F,G\right) \,\,\;  = \; \,\,  \Cy\;\;\Ay(F,G).
  \]
\end{theorem*}


\begin{example} \label{ex:arbitrary-sym0}
    Let $F \in \mathbb{R}^{3 \times 3}$ and $G \in \mathbb{R}^{2 \times 2}$ be real symmetric matrices such that their corresponding eigenvalues are
    \[
      \alpha = (0,1,1), \qquad \beta = (1,2).
    \]
    Pictorially, the eigenvalue arrangement is as shown below.
    \[
      \begin{tikzpicture}
        \draw (-1, 0) -- (3, 0);
        \fill [red] (0,0) circle(3pt) node[label=$\alpha_1$]{}; \fill [red]
        (1,0) circle(3pt) node[label=$\alpha_2$]{}; \fill [red] (1,0.75)
        circle(3pt) node[label=$\alpha_3$]{}; \fill [blue] (1,1.5)
        circle(3pt) node[label=$\beta_1$]{}; \fill [blue] (2,0)
        circle(3pt) node[label=$\beta_2$]{};
      \end{tikzpicture} \]
    In Example~\ref{ex:g1}, we computed directly from Definition~\ref{def:EC} that
    \[
      \EC(F,G) = \begin{bmatrix}
                 1/4 \\ 1/2 \\ 5/4.
               \end{bmatrix}
      \]
      On the other hand, we compute
      \begin{align*}
        \Cy &=
              \left[\begin{array}{ccc}
-\frac{1}{4} & \frac{1}{4} & \frac{3}{4} 
\\
 \frac{1}{4} & \frac{1}{4} & -\frac{3}{4} 
\\
 \frac{1}{4} & -\frac{1}{4} & \frac{1}{4} 
        \end{array}\right] \\
        \Ay(F,G) &=
                   \begin{bmatrix}
                     \overline{v}(D_1) \\
                     \overline{v}(D_2) \\
                     \overline{v}(D_3) 
                   \end{bmatrix} =
              \begin{bmatrix}
                \overline{v} \left(x^6 - 5x^5 + 9x^4 - 7x^3 + 2x^2 \right) \\
                \overline{v} \left( x^6 + 5x^5 + 8x^4 + 4x^3 \right) \\
                \overline{v} \left( x^2 - 2x \right)
              \end{bmatrix}.
      \end{align*}
      For example, consider $\overline{v} \left( x^6 + 5x^5 + 8x^4 + 4x^3 \right)$.
      We have
      \begin{align*}
        \overline{v}\left( x^6 + 5x^5 + 8x^4 + 4x^3 \right) &= v(x^6 + 5x^5 + 8x^4 + 4x^3) + \frac{\tdeg(x^6 + 5x^5 + 8x^4 + 4x^3)}{2} \\
                                                            &= 0 + \frac{3}{2} \\
                                                            &= 3/2.
      \end{align*}
      Repeating this process for the other two entries of $\Ay(F,G)$, we find that
      \begin{align*}
        \Ay(F,G) &= \begin{bmatrix}
          5 \\
          3/2 \\
          3/2
              \end{bmatrix}.
      \end{align*}
      Observe now that
      \[
        \Cy \Ay(F,G) \;\;\;=\;\;\; 
              \left[\begin{array}{ccc}
-\frac{1}{4} & \frac{1}{4} & \frac{3}{4} 
\\
 \frac{1}{4} & \frac{1}{4} & -\frac{3}{4} 
\\
 \frac{1}{4} & -\frac{1}{4} & \frac{1}{4} 
        \end{array}\right] 
        \begin{bmatrix}
          5 \\
          3/2 \\
          3/2
        \end{bmatrix} \\
        \;\;\;=\;\;\; 
          \begin{bmatrix}
            1/4 \\ 1/2 \\ 5/4
          \end{bmatrix} \\
                \;\;\;=\;\;\;  \EC(F,G).
      \]
      \exend
\end{example}

\begin{example}
    Let $F \in \mathbb{R}^{3 \times 3}$ and $G \in \mathbb{R}^{3 \times 3}$ be real symmetric matrices such that their corresponding eigenvalues are
    \[
      \alpha = (0,0,0), \qquad \beta = (0,0,0).
    \]
    Pictorially, the eigenvalue arrangement is as shown below.
    \[
      \begin{tikzpicture}
        \draw (-1, 0) -- (1, 0);
        \fill [red] (0,0) circle(3pt) node[label=$\alpha_1$]{};
        \fill [red]
        (0,1) circle(3pt) node[label=$\alpha_2$]{};
        \fill [red] (0,2)
        circle(3pt) node[label=$\alpha_3$]{};
        \fill [blue] (0,3)
        circle(3pt) node[label=$\beta_1$]{};
        \fill [blue] (0,4)
        circle(3pt) node[label=$\beta_2$]{};
        \fill [blue] (0,5.0)
        circle(3pt) node[label=$\beta_3$]{};
      \end{tikzpicture} \]
    Using Definition \ref{def:EC}, one can compute that
    \[
      \EC(F,G) = \begin{bmatrix}
        9/8 \\ 9/8 \\ 3/8
               \end{bmatrix}.
      \]
      On the other hand, we compute
      \begin{align*}
        \Cy &=
              \left[\begin{array}{ccc}
-\frac{1}{4} & \frac{1}{4} & \frac{3}{4} 
\\
 \frac{1}{4} & \frac{1}{4} & -\frac{3}{4} 
\\
 \frac{1}{4} & -\frac{1}{4} & \frac{1}{4} 
        \end{array}\right] \\
        \Ay &=
              \begin{bmatrix}
                \overline{v} \left(x^9 \right) \\
                \overline{v} \left( x^9  \right) \\
                \overline{v} \left( x^3 \right)
              \end{bmatrix}.
      \end{align*}
      For example, consider $\overline{v} \left( x^9 \right)$.
      We have
      \begin{align*}
        \overline{v}\left( x^9 \right) &= v(x^9) + \frac{\tdeg(x^9)}{2} \\
                                                            &= 0 + 9/2 \\
                                       &= 9/2.
      \end{align*}
      Repeating this process for the other two entries of $\Ay$, we find that
      \begin{align*}
        \Ay &= \begin{bmatrix}
          9/2 \\ 9/2 \\ 3/2
              \end{bmatrix}.
      \end{align*}
      Observe now that
      \[
        \Cy \Ay \;\;\;=\;\;\; 
              \left[\begin{array}{ccc}
-\frac{1}{4} & \frac{1}{4} & \frac{3}{4} 
\\
 \frac{1}{4} & \frac{1}{4} & -\frac{3}{4} 
\\
 \frac{1}{4} & -\frac{1}{4} & \frac{1}{4} 
        \end{array}\right] 
        \begin{bmatrix}
          9/2 \\ 9/2 \\ 3/2
        \end{bmatrix} \\
        \;\;\;=\;\;\;
          \begin{bmatrix}
            9/8 \\ 9/8 \\ 3/8
          \end{bmatrix} \\
                \;\;\;=\;\;\; \EC(F,G).
      \]
      \exend
\end{example}

  To prove Theorem \ref{thm:arbitrary}, we take the following approach.

\begin{enumerate}
\item First, in Lemma \ref{lem:decomp} we reduce to the case where $n=1$ (i.e., where $G$ is a $1 \times 1$ matrix).
\item Second, in Lemma \ref{lem:single-sym} we prove the above theorem in the case where $n=1$.
\item Finally, in Lemma \ref{lem:combine-sym}, we build back up from the $n=1$ case to arbitrary $n$.
  The proof is then complete.
\end{enumerate}
Our proof strategy is illustrated by the following diagram.
\[
\begin{array}{|ccc|}
  \hline
  & &\\
  \EC(F,G) & \overset{{\color{green}?}}{{\color{green}=\mathrel{\mkern-3mu}=}} & \Cy\,\,\Ay(F,G) \\[1.5ex]
  \text{Lemma \ref{lem:decomp}}\hspace{1em}\Big\Vert \hspace{6em}        &  & \hspace{6em} \Big\Vert \hspace{1em} \text{Lemma \ref{lem:combine-sym}} \\[1.5ex]
  \sum_{j=1}^{n} \EC(F, [\beta_j]) &  =\mathrel{\mkern-3mu}= & \sum_{j=1}^{n} \Cy \Ay(F, [\beta_j])\\[1.5ex]
   & \text{Lemma \ref{lem:single-sym}} & \\
  &&\\
                                                                            \hline
\end{array}
\]

First, we reduce to the case where $n = 1$ (i.e., where $G$ is a $1 \times 1$ matrix).

\begin{lemma}[Decompose $\EC$ into eigenvalues]
    \label{lem:decomp}
    For all real symmetric $F$~and~$G$ we have
    \[
      \EC(F,G) = \sum_{j=1}^{n} \EC(F, [\beta_j]).
    \]
\end{lemma}

  \begin{proof}[Proof of Lemma~\ref{lem:decomp}]
    We break the proof into two cases.
    \begin{description}
    \item[Case 1: $F$~and~$G$ are generic.] Recall Definition~\ref{def:GEC}:
      \begin{align*}
        \EC(F,G) = \GEC(F,G) = (c_1, \dots, c_m) \text{ where  }{c_t = \# \{j : \beta_j \in A_t\}}.
      \end{align*}
      Note that 
      \[
        c_t \;\;\;=\;\;\; \# \{j : \beta_j \in A_t\} \\
            \;\;\;=\;\;\; \sum_{j=1}^{n}
              \begin{cases}
                1 & \text{ if $\beta_j \in A_t$} \\
                0 & \text{ else},
              \end{cases}
      \]
\noindent      and the summand is exactly the $t$-th component of $\EC(F, [\beta_j])$.

    \item[Case 2: $F$~and~$G$ are not generic.] Recall Definition~\ref{def:EC}:
      \begin{align*}
        \EC(F,G) &= \frac{1}{2^m} \sum_{d \in \{-1,1\}^m} \GEC(F_d, G).
      \end{align*}
      Note that the pair $F_d$ and $G$ is generic.
      Then we have
      \begin{alignat*}{2}
        \EC(F,G) &= \frac{1}{2^m} \sum_{d \in \{-1,1\}^m} \GEC(F_d, G) && \text{by Definition~\ref{def:EC}} \\
                &= \frac{1}{2^m} \sum_{d \in \{-1,1\}^m} \sum_{j=1}^{n} \GEC(F_d, [\beta_j]) && \text{by Case 1}\\
                &=  \sum_{j=1}^{n} \underbrace{\frac{1}{2^m} \sum_{d \in \{-1,1\}^m} \GEC(F_d, [\beta_j])}_{\text{$=\EC(F, [\beta_j])$ by Definition~\ref{def:EC}}} \qquad \\
                &= \frac{1}{2^m} \sum_{j=1}^{n} \EC(F, [\beta_j]).
      \end{alignat*}
    \end{description}
  \end{proof}

Next, we prove the above theorem in the case where $n=1$.
\begin{lemma}
  \label{lem:single-sym}
  Let $F \in \mathbb{R}^{m \times m}$ be arbitrary and let $\beta \in \mathbb{R}$.
  Let $r \in [m]$ be arbitrary.
  Then we have
  \[
    \EC(F, [\beta]) = \Cy \Ay (F, [\beta]).
    \]
    \end{lemma}

    \begin{example}
      Let $F = \diag(0,1,1)$ and let $\beta = 1$.
      One can use Definition \ref{def:EC} to compute
      \[
        \EC(F, [\beta]) =
        \begin{bmatrix}
          1/2 \\ 1/4 \\ 1/2
        \end{bmatrix}.
      \]
      On the other hand, using the definitions of $\Cy$ (from Definition \ref{def:Cy}) and $\Ay$ (from Definition \ref{def:A}), we have 
      \begin{align*}
        \Cy \Ay (F, [\beta]) &=
                           \left[\begin{array}{ccc}
-\frac{1}{4} & \frac{1}{4} & \frac{3}{4} 
\\
 \frac{1}{4} & \frac{1}{4} & -\frac{3}{4} 
\\
 \frac{1}{4} & -\frac{1}{4} & \frac{1}{4} 
\end{array}\right]
                           \begin{bmatrix}
                             \overline{v}(D_1) \\
                             \overline{v}(D_2) \\
                             \overline{v}(D_3)
                           \end{bmatrix}
      \end{align*}
      Note that for the purposes of easier numeric computation and reasoning, we can replace $D_r$ in the above with~$h_r$ since we know the eigenvalues of $F$ and $G = [\beta]$.
      With this we can compute

        \[
        \begin{bmatrix}
          \overline{v}(D_1) \\
          \overline{v}(D_2) \\
          \overline{v}(D_3)
        \end{bmatrix} 
        \;=\;
        \begin{bmatrix}
          \overline{v}(h_1) \\
          \overline{v}(h_2) \\
          \overline{v}(h_3)
        \end{bmatrix}  \\
        \;=\;
        \begin{bmatrix}
          {v}(h_1) + \frac{\tdeg(h_1)}{2} \\
          {v}(h_2) + \frac{\tdeg(h_2)}{2} \\
          {v}(h_3) + \frac{\tdeg(h_3)}{2}
        \end{bmatrix}  \\
        \;=\;
                        \begin{bmatrix}
                          v(x^3-x^2) + \frac{\tdeg(x^3 - x^2)}{2} \\
                          v(x^3) + \frac{\tdeg(x^3)}{2} \\
                          v(x) + \frac{\tdeg(x)}{2}
                        \end{bmatrix} \\
        \;=\;
                        \begin{bmatrix}
                          1 + 2/{2} \\
                          0 + {3}/{2} \\
                          0 + {1}/{2}
                        \end{bmatrix} \\
        \;=\;
                        \begin{bmatrix}
                          2 \\ {3}/{2} \\ {1}/{2}
                        \end{bmatrix},
      \]
      and therefore
      \[
        \Cy \Ay(F, [\beta]) \;\;=\;\; 
                           \left[\begin{array}{ccc}
-\frac{1}{4} & \frac{1}{4} & \frac{3}{4} 
\\
 \frac{1}{4} & \frac{1}{4} & -\frac{3}{4} 
\\
 \frac{1}{4} & -\frac{1}{4} & \frac{1}{4} 
\end{array}\right]
                        \begin{bmatrix}
                          2 \\ {3}/{2} \\ {1}/{2}
                        \end{bmatrix} \\
                        \;\;=\;\; 
                          \begin{bmatrix}
                            1/2 \\ 1/4 \\ 1/2
                          \end{bmatrix} \\
                        \;\;=\;\; \EC(F, [\beta]).
      \]
      \exend
    \end{example}

    \begin{proof}[Proof of Lemma \ref{lem:single-sym}]
      If $F$ does not have $\beta$ as an eigenvalue (i.e. $F$ and $[\beta]$ are a generic pair), then the lemma follows trivially by Theorem \ref{thm:generic}.
      Thus, for the remainder of the proof, suppose that $\beta$ is an eigenvalue of $F$.
      By definition \ref{def:EC}, the left-hand side of the lemma claim is
      \begin{alignat*}{2}
        \EC(F, [\beta])
        &= \frac{1}{2^m}\sum_{d \in \{-1, 1\}^m} \GEC(F_d, [\beta]) \\
        &= \frac{1}{2^m}\sum_{d \in \{-1, 1\}^m} \Cy \Ay (F_d, [\beta]) && \text{ since $F_d$ and $[\beta]$ are a generic pair} \\
        &=\Cy \frac{1}{2^m}\sum_{d \in \{-1, 1\}^m}  \Ay (F_d, [\beta]) \qquad && \text{by linearity}.
      \end{alignat*}
      Let $h_{r,d}$ (resp. $D_{r,d})$ be the polynomial constructed via Definition \ref{def:A} for the perturbed matrix~$F_d$ and the matrix $G = [\beta]$; that is,
      \begin{align*}
        h_{r,d}&= \prod_{\substack{I \subset [m]\\ \#I = r }} \left(x+\prod_{i \in I}\left(  (\alpha_{i} + \varepsilon d_i) -\beta\right)  \right),
      \end{align*}
      and respectively for $D_{r,d}$ (being those polynomials $h_{r,d}$ expressed in terms of the entries of $F$ and $G$ using the FTSP).
      Then for $r \in [m]$ we have
      \begin{alignat*}{2}
        (\Ay (F_d, [\beta]))_r &= v(D_{r,d}) \\
                           &= v(h_{r,d}) \qquad && \text{since by Def. \ref{def:A} $D_{r,d}(p) = h_{r,d}(\alpha, \beta)$}.
      \end{alignat*}
      Now, we examine the right-hand side of the lemma claim.
      Recall that $\Ay(F,G)$ is defined differently for generic pairs versus arbitrary pairs of matrices.
      Note that for $r \in [m]$ we have
      \begin{alignat*}{2}
        (\Ay(F, [\beta]))_r &= \overline{v}(D_r) \\
                        &= v(D_r) + \frac{\tdeg(D_r)}{2} \qquad && \text{by the definition in Theorem \ref{thm:arbitrary}} \\
                        &= v(h_r) + \frac{\tdeg(h_r)}{2} && \text{by the definition of $D_r$}.
      \end{alignat*}
      Let $r \in [m]$ be arbitrary but fixed. To prove the lemma, it therefore suffices to show that
        \begin{align*}
    \frac{1}{2^m} \sum_{d \in \{-1, 1\}^m} v(h_{r,d}) = v(h_r) + \frac{\tdeg(h_r)}{2}.
        \end{align*}
        We approach this proof by studying the left-hand side of the above.

        \medskip
        
      \noindent Recall that
      \begin{align*}
        h_{r,d}&= \prod_{\substack{I \subset [m]\\ \#I = r }} \left(x+\prod_{i \in I}\left(  (\alpha_{i} + \varepsilon d_i) -\beta\right)  \right).
      \end{align*}
      Let $M = \{i : \alpha_i = \beta\}$ and let 
      \begin{align*}
        h^=_{r,d}&= \prod_{\substack{I \subset [m] \\\#I = r \\ I \cap M \ne \emptyset  }} \left(x+\prod_{i \in I}\left(  (\alpha_{i} + \varepsilon d_i) -\beta\right)  \right) \\
        h^{\ne}_{r,d}&= \prod_{\substack{I \subset [m]\\\#I = r \\ I \cap M = \emptyset  }} \left(x+\prod_{i \in I}\left(  (\alpha_{i} + \varepsilon d_i) -\beta\right)  \right) .
      \end{align*}
      Then, clearly we have $h_{r,d} = h^{\ne}_{r,d} h^=_{r,d}.$
      Informally speaking, the polynomial $h^{\ne}_{r,d}$ comprises the terms of $h_{r,d}$ which are invariant under perturbation, because it only involves $\alpha$'s which do not coincide with $\beta$. The polynomial $h^=_{r,d}$ collects all the terms of $h_{r,d}$ which contain at least one $\alpha$ which coincides with $\beta$, and therefore comprises all terms which are affected by perturbation and depend on $d$.

     \medskip

      \noindent For each $I \subset [m]$ with $\#I = r$ and $I \cap M \ne \emptyset$, let
      \begin{align*}
        I_- &= \{i \in I : \alpha_i < \beta \} \\
        I_0 &= \{i \in I : \alpha_i = \beta \} \\
        I_+ &= \{i \in I : \alpha_i > \beta \}.
      \end{align*}
      Then, by construction $I_0 \ne \emptyset$ and $I = I_- \uplus I_0 \uplus I_+$ and so we have
      \begin{align*}
        h^=_{r,d}&= \prod_{\substack{I \subset [m]\\\#I = r \\ I \cap M \ne \emptyset  }} \left(x+\prod_{i \in I_-}\left(  (\alpha_{i} + \varepsilon d_i) -\beta\right) \prod_{i \in I_+}\left(  (\alpha_{i} + \varepsilon d_i) -\beta\right)  \prod_{i \in I_0}\left(  (\alpha_{i} + \varepsilon d_i) -\beta\right)\right).
      \end{align*}
      Let
      \begin{align*}
        p(h^=_{r,d}) &= \# \text{ positive roots of }h^=_{r,d}.
      \end{align*}
      Then, from the above we have
      \begin{align*}
        p(h^=_{r,d}) &= \# \{ I : -(-1)^{\# I_-} (1)^{\# I_+} \prod_{i \in I_0} d_i = 1 \}.
      \end{align*}
      Note that since $h^=_{r,d}$ is a polynomial with only real roots, by Descartes' rule of signs we have that $v(h^=_{r,d}) = p(h^=_{r,d})$.
      Hence it suffices to show that
      \[
        \frac{1}{2^m} \sum_{d \in \{-1, 1\}^m} p(h_{r,d}) = p(h_{r,d}) + \frac{ \tdeg(h_{r,d})}{2}.
      \]
      We first prove the following claim.

      \medskip
      
      \noindent\textbf{Claim:} We have
      $$\frac{1}{2^m} \sum_{d \in \{-1,1\}^m} p(h^=_{r,d}) = \underset{d \in \{-1,1\}^m}{\operatorname*{average}} \,\, p(h^=_{r,d}) \,\, = \,\, \frac{\deg h^=_{r,d}}{2}.$$

      \medskip

      \noindent \textit{Proof of claim:}
      For a concise notation, let $S_{I, d} = -(-1)^{\# I_-} (1)^{\# I_+} \prod_{i \in I_0} d_i$.
      With this notation, we then have
      \begin{equation} \label{eq:psid}
        p(h^=_{r,d}) = \# \{ I : S_{I,d} = 1 \}.
      \end{equation}
      Note that since $\alpha_i + \varepsilon d_i \ne \beta$ for all $i$, we have that $h^=_{r,d}$ has only nonzero real roots.
      Hence, for each $I$ and $d$, we have $S_{I,d} = \pm 1$.
      Thus, to prove the claim, it suffices to show that
      \[
        \# \{(I, d) : S_{I,d} = 1\} \,\,\, =  \,\,\,
        \# \{(I, d) : S_{I,d} = -1\}.
      \]
      It suffices to show that for every fixed $I$ and $d$, there exists a unique $d'$ so that
      \[
        S_{I,d} = -S_{I, d'}.
      \]
      Now let
      \begin{align*}
        I &\subset [m] \text{ such that $\#I = r$ and $I \cap M \ne \emptyset$} \\
        d &\in \{-1,1\}^m
      \end{align*}
      be fixed.
      Recall that $I_0 = \{i : \alpha_i = \beta \} = I \cap M \ne \emptyset$.
      Now let $k$ be the smallest element of $I_0$.
      Then let
      \[
        d' = (d_1, \dots, d_{k-1}, -d_k, d_{k+1}, \dots, d_m);
      \]
      i.e., let $d'$ be a copy of $d$ but with the sign of the $k$-th entry flipped.
      Obviously, the map $d \mapsto d'$ is a 1-1 correspondence.
      Further, clearly $S_{I,d'} = -S_{I, d}$, since $d$ and $d'$ differ by exactly one sign in the $k$th slot, and $k \in I_0$. 
      Hence, we have that
      \begin{equation*}
        \# \{(I, d) : S_{I,d} = 1\} \,\,\, =  \,\,\,
        \# \{(I, d) : S_{I,d} = -1\};
      \end{equation*}
      in other words, for fixed $I$, we have that exactly half of the $d$ vectors give $S_{I,d} = 1$ while the other half give $-1$.
      Since the size of the set $\{-1,1\}^m$ is $2^m$, this means that for fixed $I$, we have
      \begin{equation}
        \label{eq:sid}
        \#\{ d \in \{-1,1\}^m : S_{I,d} = 1 \} = \frac{1}{2} 2^m = 2^{m-1}.
      \end{equation}
      Now, note that
      \begin{alignat*}{5}
        \frac{1}{2^m} \sum_{d \in \{-1,1\}^m} p(h^=_{r,d}) &=
                                                        \frac{1}{2^m} &&\sum_{d \in \{-1,1\}^m} &&\# \{I : S_{I,d} = 1\}  && && \text{by (\ref{eq:psid}) } \\
                                                      &= \frac{1}{2^m} &&\sum_{d \in \{-1,1\}^m} \;\; &&\sum_{\substack{I \subset [m]\\ \#I = r\\ I \cap M \ne \emptyset}} &&\delta_{S_{I,d}, 1} && \text{counting becomes sum} \\
                                                      &= \frac{1}{2^m} &&\sum_{\substack{I \subset [m]\\ \#I = r\\ I \cap M \ne \emptyset}} \;\; &&\sum_{d \in \{-1,1\}^m} &&\delta_{S_{I,d}, 1} \qquad && \text{switching the order of summation}\\
                                                      &= \frac{1}{2^m} &&\sum_{\substack{I \subset [m]\\ \#I = r\\ I \cap M \ne \emptyset}} \;\; &&2^{m-1} &&&& \text{by (\ref{eq:sid}) } \\
                                                      &= \frac{1}{2}  &&\sum_{\substack{I \subset [m]\\ \#I = r\\ I \cap M \ne \emptyset}} &&1 \\
                                                      &= \frac{1}{2} &&\deg h^=_{r,d}.
      \end{alignat*}
      Note that the key step in the above was switching the order of summation and summing over $d$ inside first.
      This is because it is very difficult to reason when the summation over $I$ is inside (i.e. reasoning with fixed~$d$); it is much easier to reason that, for fixed $I$, exactly half of the pairs~$(I,d)$ have $S_{I,d} = 1$. We explain this idea more in depth in Example \ref{ex:sid}.
      With this, the claim is proved.

      \medskip

      To finish the proof of the lemma, recall that 
        $$h^{\ne}_{r,d}= \prod_{\substack{I \subset [m] \\ \#I = r \\ I \cap M = \emptyset  }} \left(x+\prod_{i \in I}\left(  (\alpha_{i} + \varepsilon d_i) -\beta\right)  \right).$$
        Note that since $h^{\ne}_{r,d}$ involves only $\alpha$'s which do not coincide with $\beta$, by a sufficiently small choice of $\varepsilon$ we have that
        \[
          \sign((\alpha_i + \varepsilon d_i) - \beta) = \sign(\alpha_i - \beta).
        \]
        Therefore
        \begin{alignat*}{2}
          p(h^{\ne}_{r,d}) &= v(h^{\ne}_{r,d}) \qquad && \text{by Descartes' rule of signs} \\
                         &= v(h_{r}) && \text{by the discussion above} \\
                         &= p(h_r) && \text{by Descartes' rule of signs, again.}
        \end{alignat*}
        Putting everything together, we therefore have
      \begin{alignat*}{2}
       \frac{1}{2^m} \sum_{d \in \{-1,1\}^m} v(h_{r,d}) &= \frac{1}{2^m} \sum_{d \in \{-1,1\}^m} p(h_{r,d}) && \text{by Descartes' rule of signs} \\
                   &=\frac{1}{2^m} \sum_{d \in \{-1,1\}^m}  p(h^{\ne}_{r,d}) + p(h^=_{r,d}) \qquad && \text{since $h_{r,d} = h^{\ne}_{r,d} h^=_{r,d}$} \\
                                                   &=\frac{1}{2^m} \sum_{d \in \{-1,1\}^m}  p(h_{r}) + p(h^=_{r,d}) && \text{since $p(h^{\ne}_{r,d}) = p(h_r)$, by the discussion above} \\
                                                   &= p(h_{r}) + 
                                                   \frac{1}{2^m} \sum_{d \in \{-1,1\}^m} p(h^=_{r,d}) && \text{since $p(h_r)$ does not depend on $d$}\\
                                                   &= p(h_{r}) + 
                                                     \frac{\deg h^=_{r,d}}{2} && \text{by the claim}. \\
        \end{alignat*}
        If we (by slight abuse of notation) substitute $d = (0, \dots, 0)$ into the expression for $h_{r,d}$, we find that
        \begin{alignat*}{2}
          \tdeg (h_{r}) &= \tdeg (h_{r,0}) \\
                        &= \tdeg \left( h^{\ne}_{r,0} \right) + \tdeg \left( h^=_{r,0} \right) \qquad && \text{since $h_{r,0} = h^{\ne}_{r,0} h^=_{r,0}$} \\
                        &= 0 + \tdeg \left( h^=_{r,0} \right) && \text{since all roots of $h^{\ne}_{r,0}$ are nonzero} \\
                        &= 0 + \deg \left( h^=_{r,d} \right) && \text{ $\forall d$, since plugging in  $d \ne 0$ preserves degree} \\
                        &= \deg (h^=_{r,d}).
        \end{alignat*}
        Finally, if we continue from above, we have
        \begin{alignat*}{2}
          p(h_{r}) + \frac{\deg h^=_{r,d}}{2}
                                                   &= p(h_r) + \frac{\tdeg(h_r)}{2} \qquad && \text{by the above} \\
                                                   &= v(h_r) + \frac{\tdeg(h_r)}{2} && \text{by Descartes' rule of signs} \\
                                                   &= \overline{v}(h_r).
      \end{alignat*}
      Thus we have shown that
        \begin{align*}
    \frac{1}{2^m} \sum_{d \in \{-1, 1\}^m} v(h_{r,d}) = v(h_r) + \frac{\tdeg(h_r)}{2},
        \end{align*}
        which means that
  \[
    \EC(F, [\beta]) = \Cy \Ay (F, [\beta]).
    \]
      Thus the lemma is proved.
    \end{proof}
    \begin{example} \label{ex:sid}
      Since the proof of Lemma \ref{lem:single-sym} is quite long and complicated, let us provide an example to better illustrate the key idea which led to the above proof.
      The crucial observation in the above proof was (\ref{eq:sid}); that is, over all $I$ and $d$, exactly half give $S_{I,d} = 1$.  

      Let us consider an example where $F = \diag(0,0,1)$ and $\beta = 0$.
      Using the notation established in the proof above, we have $m = 3$ and $M = \{i : \alpha_i = \beta\} = \{1, 2\}$.
      We will consider the case where~$r = 2$.
      Note
      \begin{align*}
        h^=_{2,d} = &\prod_{\substack{I \subset [3] \\ \# I = 2 \\ I \cap \{1,2\} \ne \emptyset}} \left( x + \prod_{i \in I} (\alpha_i + \varepsilon d_i - \beta) \right) \\
                 =
                   &\,\,\underbrace{(x + (\alpha_1 + \varepsilon d_1 - \beta) (\alpha_2 + \varepsilon d_2 - \beta))}_{I = \{1,2\}} \\
                   &\,\,\underbrace{(x + (\alpha_1 + \varepsilon d_1 - \beta) (\alpha_3 + \varepsilon d_3 - \beta))}_{I = \{1,3\}} \\
                   &\,\,\underbrace{(x + (\alpha_2 + \varepsilon d_2 - \beta) (\alpha_3 + \varepsilon d_3 - \beta))}_{I = \{2,3\}}.
      \end{align*}
      Note that each of the three terms in $h^=_{2,d}$ is labelled with the corresponding $I$.
      For each $I$, we compute the corresponding sets $I_-$, $I_0$, and $I_+$.

      \begin{center}
        \begin{tabular}{|c|c|c|c|}
          \hline
          $I$&$I_-$&$I_0$&$I_+$ \\
          \hline
          $\{1,2\}$ & $\emptyset$ & $\{1,2\}$ & $\emptyset$\\
          $\{1,3\}$ & $\emptyset$ & $\{1\}$ & $\{3\}$\\
          $\{2,3\}$ & $\emptyset$ & $\{2\}$ & $\{3\}$\\
          \hline
        \end{tabular}
      \end{center}
    Now, consider $d = (d_1, d_2, d_3) \in \{-1, 1\}^3$.
    Recall from the above proof that $S_{I,d} = - (-1)^{\# I_-} (1)^{\# I_+} \prod_{i \in I_0} d_i$.
    We compute
    \begin{alignat*}{5}
      S_{12, d} &= -(-1)^0 \,\, && \,\,(1)^0\,\,&&\,\,d_1d_2\,\, &&= &&-d_1d_2 \\
      S_{13, d} &= -(-1)^0&&(1)^1&&d_1 &&= &&-d_1 \\
      S_{23, d} &= -(-1)^0&&(1)^1&&d_2 &&= &&-d_2.
    \end{alignat*}
    Finally, with this information, we now compute $S_{I,d}$ for each $I$ enumerated above and for each~$d \in \{-1,1\}^3$.
    
\noindent The table below is constructed as follows. 
\begin{itemize}
\item The rows are indexed by the $I$ sets, where the leftmost column shows the decompositions $I_-$, $I_0$, and~$I_+$, the disjoint union of which is $I$.
\item The columns are indexed by~$d \in \{-1,1\}^m$.
\item Each interior cell contains the entries of the~$d$ vector which correspond to the indices in $I_0$. 
\item Immediately below in the same cell is the sign of~$S_{I,d}$. 
\end{itemize}  
  
\noindent For example, the top-left-most cell corresponds to $I_0 \{1,2\}$ and $d = - - - $, and so the cell contains~$d_1, d_2 = -, -$, and beneath it is the sign of~$S_{I,d} = -$ colored red. (Positive $S_{I,d}$ are colored green.)

\medskip
    
\noindent Each interior cell is also color-paired with the unique cell in the same row whose $d$ vector is obtained by reversing the sign of $d_k$, where $k$ is the smallest index in $I_0$, as in the proof.
For example, in the top-left-most cell the first entry $d_1 = -$ is colored cyan and is paired with the corresponding cell in the same row which has $d_1 = +$.

    \begin{center}
      \begin{tabular}{|c|c|c|c|c|c|c|c|c|}
        \hline
        $I \setminus d$ & $ - - - $& $- - +$&$- + -$& $- + +$&$+ - - $&$+ - +$& $+ + -$& $+ + +$ \\ \hline
        $
        \begin{array}{c}
          I_- = \emptyset \\
          I_0 = 12 \\
          I_+ = \emptyset
        \end{array}
        $ &
            $\begin{array}{c}
             d_1, d_2 = {\color{cyan}-}, - \\ S_{I,d} = {\color{red} - } 
            \end{array}$ & 
            $\begin{array}{c}
             {\color{blue}-}, - \\ {\color{red} - } 
            \end{array}$ & 
            $\begin{array}{c}
             {\color{red}-}, + \\ {\color{green} + } 
            \end{array}$ & 
            $\begin{array}{c}
             {\color{magenta}-}, + \\ {\color{green} + } 
            \end{array}$ & 
            $\begin{array}{c}
             {\color{cyan}+}, - \\ {\color{green} + } 
            \end{array}$ & 
            $\begin{array}{c}
             {\color{blue}+}, - \\ {\color{green} + } 
            \end{array}$ & 
            $\begin{array}{c}
             {\color{red}+}, + \\ {\color{red} - } 
            \end{array}$ & 
            $\begin{array}{c}
             {\color{magenta} + }, + \\ {\color{red} - } 
            \end{array}$ 
        \\ \hline
        $
        \begin{array}{c}
          I_- = \emptyset \\
          I_0 = 1 \\
          I_+ = 3 
        \end{array}
        $ &
            $\begin{array}{c}
             d_1 = {\color{cyan}-} \\ S_{I,d} = {\color{green} + } 
            \end{array}$ & 
            $\begin{array}{c}
             {\color{blue}-} \\ {\color{green} + } 
            \end{array}$ & 
            $\begin{array}{c}
             {\color{red}-} \\ {\color{green} + } 
            \end{array}$ & 
            $\begin{array}{c}
             {\color{magenta}-} \\ {\color{green} + } 
            \end{array}$ & 
            $\begin{array}{c}
             {\color{cyan}+} \\ {\color{red} - } 
            \end{array}$ & 
            $\begin{array}{c}
             {\color{blue}+} \\ {\color{red} - } 
            \end{array}$ & 
            $\begin{array}{c}
             {\color{red}+} \\ {\color{red} - } 
            \end{array}$ & 
            $\begin{array}{c}
             {\color{magenta} + }\\ {\color{red} - } 
            \end{array}$ 
        \\ \hline
        $
        \begin{array}{c}
          I_- = \emptyset \\
          I_0 = 2 \\
          I_+ = 3 
        \end{array}
        $ &
            $\begin{array}{c}
             d_2 = {\color{cyan}-} \\ S_{I,d} = {\color{green} + } 
            \end{array}$ & 
            $\begin{array}{c}
             {\color{blue}-} \\ {\color{green} + } 
            \end{array}$ & 
            $\begin{array}{c}
             {\color{cyan}+} \\ {\color{red} - } 
            \end{array}$ & 
            $\begin{array}{c}
             {\color{blue}+} \\ {\color{red} - } 
            \end{array}$ & 
            $\begin{array}{c}
             {\color{red}-} \\ {\color{green} + } 
            \end{array}$ & 
            $\begin{array}{c}
             {\color{magenta}-} \\ {\color{green} + } 
            \end{array}$ & 
            $\begin{array}{c}
             {\color{red}+} \\ {\color{red} - } 
            \end{array}$ & 
            $\begin{array}{c}
             {\color{magenta} + }\\ {\color{red} - } 
            \end{array}$ 
        \\ \hline
      \end{tabular}
    \end{center}
    
    The key takeaway of the above table is as follows. Observe that in each row, exactly half the cells have a green plus in the bottom entry (the $S_{I,d}$ value) and half have a red minus.
    This shows that $$\#\{(I,d) : S_{I,d} = 1 \} \;\;\;\;=\;\;\;\; \# \{(I,d) : S_{I,d} = -1\}.$$
    Further, in each row, there is exactly one cyan minus and one cyan plus, and respectively for blue, red, and magenta.
    Each colored pair represents the map $d \mapsto d'$ discussed in the proof, which is obtained by reversing the sign of the entry of $d$ corresponding to the smallest index in $I_0$.

    Finally, observe that summing over all $(I,d)$ corresponds to simply summing over the entire table.
    If we sum down the rows first (i.e., summing over $I$ with $d$ held constant first), it is unclear what the pattern is, if any. However, if we instead sum across the columns (i.e., summing over $d$ with $I$ held constant), it is easy to see that exactly half the cells have $S_{I,d} = 1$ and the other half have~$S_{I,d} = -1$. This corresponds exactly with switching the order of summation in the previous proof.
    \exend
  \end{example}
    
Finally, we build back up from the $n=1$ case to arbitrary $n$.
    \begin{lemma}
      \label{lem:combine-sym}
      Let $F$ and $G$ be arbitrary real symmetric matrices.
      Then
      \begin{align*}
        \sum_{j=1}^{n} \Cy \Ay (F, [\beta_j]) = \Cy \Ay(F,G).
      \end{align*}
    \end{lemma}
    \begin{proof}
      Recall that
      \begin{align*}
        (\Ay(F, [\beta_j]))_r = \overline{v}(D_{r, j}) 
      \end{align*}
      where 
      \begin{align*}
        D_{r,j} &= \prod_{I \subset [m], \#I = r} \left( x + \prod_{i \in I}^{} (\alpha_i - \beta_j) \right).
      \end{align*}
      Note that we also have
      \[
        D_{r} 
        \;\;=\;\; 
        \prod_{\substack{I \subset [m], \,\,\#I = r \\ j \in [n]}} \left( x + \prod_{i \in I}^{} (\alpha_i - \beta_j) \right)         \;\;=\;\; 
        \prod_{j \in [n]} D_{r,j}.
      \]
      Then
      \begin{alignat*}{2}
        \sum_{j=1}^{n} (\Ay(F, [\beta_j]))_r &= \sum_{j=1}^{n} \overline{v}(D_{r,j}) \\
                                      &= \sum_{j=1}^{n} v(D_{r,j}) + \frac{\tdeg (D_{r,j})}{2} \\
                                      &= \sum_{j=1}^{n} v(D_{r,j}) + \frac12 \sum_{j=1}^{n}\tdeg D_{r,j} \qquad  \\
                                      &= v(D_r) + \frac{\tdeg D_r}{2} && \text{since $D_r = \prod_{j=1}^{n} D_{r,j}$} \\
                                      &= \overline{v}(D_r),
      \end{alignat*}
      which is exactly the $r$-th element of $\Ay(F,G)$.
      Therefore, we have
      \[
        \sum_{j=1}^{n} \Cy \Ay (F, [\beta_j]) \;\;\;=\;\;\;
        \Cy \sum_{j=1}^{n} \Ay (F, [\beta_j]) \;\;\;=\;\;\;
         \Cy \Ay (F, [\beta_j]).
      \]
      Thus the proof is complete.
    \end{proof}

  \noindent With that, we are now ready to prove our main theorem in full generality.
  \begin{proof}[Proof of Theorem~\ref{thm:main}]
    Let $F$~and~$G$ be arbitrary real symmetric matrices.
    Then

\[
  \EC(F,G) \underset{\rm Lemma~\ref{lem:decomp}}{\hspace{2em}=\hspace{2em}}  
  \sum_{j=1}^{n} \EC(F, [\beta_j])
  \underset{\rm Lemma~\ref{lem:single-sym}}{\hspace{2em}=\hspace{2em}} 
  \sum_{j=1}^{n} \Cy\,\,\Ay(F, [\beta_j]) 
  \underset{\rm Lemma~\ref{lem:combine-sym}}{\hspace{2em}=\hspace{2em}} 
  \Cy \Ay(F, G).
\]
\end{proof}


\section{Algorithms}
\label{sec:algorithms}
In this section, we render the main result (Theorem~\ref{thm:main}) into an algorithm for those  who are  interested in the
implementation. 
\begin{algorithm}
[Main: Condition for EC]\ \label{alg:main}

\begin{itemize}
\item[In: ] $F \in \mathbb{R}[p]^{m \times m}$ and  $G \in \mathbb{R}[p]^{n \times n}$, symmetric, where $p$ is a finite list of parameters

  $c \in \mathbb{R}^m$, eigenvalue configuration

\item[Out:] $P,$ quantifier-free condition on $p$ such that $c=\EC\left(  F,G\right)  $
\end{itemize}

\begin{enumerate}
\item For $r=1,\ldots,m$ do

\begin{enumerate}

\item 
$Y_{r} \longleftarrow \Big\{  \left(  i_{1},\ldots,i_{r},j\right) \,\, : \,\,1\leq i_{1}%
<\cdots<i\,_{r}\leq m\ \ \wedge\ \ 1\leq j\leq n\Big\}$

\item $\displaystyle h_{r}\longleftarrow \prod\limits_{\left(  i_{1},\ldots,i_{r},j\right)
\in Y_{r}}\left(  x+\prod_{p=1}^{r}\left(  \alpha_{i_{p}}-\beta_{j}\right)
\right) \hspace{14.83em} \in \mathbb{Z}[\alpha, \beta][x]$
\item $u_{r}\longleftarrow\operatorname*{FTSP}(h_r, \, [  \alpha_1, \dots, \alpha_m  ], \, [\beta_1, \dots, \beta_n])  \hspace{13.9em} \in \mathbb{Z}[\gamma, \delta][x]$

\item $d_r \longleftarrow u_r\Big(\left[\operatorname*{coeff}_{z^0}(\det(zI_m + F)), \dots, \operatorname*{coeff}_{z^{m-1}}(\det(xI_m + F))\right],$

  $\phantom{eeeeeeeeee} \left[\operatorname*{coeff}_{ z^0}(\det(xI_n + G)), \dots, \operatorname*{coeff}_{z^{n-1}}(\det(xI_n + G))\right], \,\, x\Big) \hspace{3.55em} \in \mathbb{Z}[p][x]$
\end{enumerate}

\item $T \longleftarrow$ matrix in $\mathbb{Z}^{m \times m}$ where $\displaystyle T_{r,s} = \sum_{t=1}^{m} \binom{m-t}{r-t} (-2)^{t-1}
\binom{s}{t}$ 
\item $y \longleftarrow T c$

\item $\displaystyle P\longleftarrow \bigwedge_{1 \le r \le m} \overline{v}(d_r) = y_r$

\item Return $P$
\end{enumerate}
\end{algorithm}

\noindent The above algorithm calls a sub-algorithm called  FTSP, which stands for Fundamental Theorem of Symmetric Polynomials. We provide here only the input/output specification.      

\begin{algorithm}[Sub: FTSP] \ \label{alg:ftsp}
\begin{itemize}
\item[In :]  $h \in \mathbb{R}[\alpha, \beta][x]$, a polynomial symmetric in variables $\alpha = (\alpha_1, \dots, \alpha_m)$ and $\beta = (\beta_1, \dots, \beta_n)$

\item[Out:] $u \in \mathbb{R}[\gamma, \delta][x]$ such that $h = u\big(\left[e_1(\alpha), \dots, e_m(\alpha)\right], \left[e_1(\beta), \dots, e_n(\beta)\right], x \big)$
\end{itemize}
\end{algorithm}

\noindent  There are many algorithms for FTSP in the literature, for instance~\cite{St93}.  




\section{Conclusion}
\label{sec:conclusion}
In this section, we summarize the contribution of this paper and discuss future directions.
\paragraph{Summary:}
In this paper, we gave an algorithm which solves the following problem: given parametric real symmetric matrices $F$ and $G$ and an eigenvalue configuration $c$, give a condition on the parameters so that $\EC(F,G) = c$.
To accomplish this, we gave a natural definition of eigenvalue configuration and gave an invertible combinatorial transformation to relate it to a set of real root counting problems of certain symmetric polynomials constructed from the eigenvalues. We then applied the Fundamental Theorem of Symmetric Polynomials to express those symmetric polynomials in terms of the parameters of the matrices~$F$ and $G$ to obtain a quantifier-free condition.

\paragraph{Future directions:}
We are investigating ways to prune and/or simplify the output condition. More precisely, recall that the output of Algorithm \ref{alg:main} can be written as a disjunction of conjunctions. In many cases, some of these conjunctive branches might be unsatisfiable by any choice of parameters, and therefore always evaluate to ``false.'' Hence, they could safely be eliminated from the output condition. Ideally, we would like to systematically remove these branches from the output condition, or avoid computing them entirely.

\bibliographystyle{plain}
\bibliography{../../../reference/refs}

\bigskip\noindent\textbf{Acknowledgements.} Hoon Hong was partially supported
by US National Science Foundation NSF-CCF-2212461. Daniel Profili's PhD work was supported by US National Science Foundation NSF-CCF-2212461. Rafael Sendra was partially
supported by the research project PID2020-113192GB-I00 (Mathematical
Visualization: Foundations, Algorithms and Applications) from the Spanish MICINN.

\end{document}